\documentclass{article}
\usepackage{amssymb,amsmath}
\usepackage{amscd}
\usepackage{amsthm}
\usepackage{latexsym}
\usepackage{enumerate}
\usepackage{changes}
\usepackage{color,xcolor}
\usepackage{citeref}
\newtheorem{theorem}{Theorem}[section]
\newtheorem{lemma}[theorem]{Lemma}
\newtheorem{corollary}[theorem]{Corollary}
\newtheorem{proposition}[theorem]{Proposition}\theoremstyle{definition}
\newtheorem{remark}[theorem]{Remark}
\newtheorem{example}[theorem]{Example}
\newtheorem{examples}[theorem]{Examples}
\newtheorem{definition}[theorem]{Definition}
\numberwithin{equation}{section}
\numberwithin{equation}{section}
\DeclareFontFamily{OT1}{pzc}{}
\DeclareFontShape{OT1}{pzc}{m}{it}{<-> s * [1.20] pzcmi7t}{}
\DeclareMathAlphabet{\mathpzc}{OT1}{pzc}{m}{it}

\newcommand{\nd}{{\ensuremath d}} 

\newcommand{\ignore}[1]{}

\newcommand{\beq}{\begin{eqnarray}}
\newcommand{\eeq}{\end{eqnarray}}
\newcommand{\beqq}{\begin{eqnarray*}}
\newcommand{\eeqq}{\end{eqnarray*}}
\newcommand{\hra}{\hookrightarrow}
\newcommand{\bit}{\begin{itemize}}
\newcommand{\eit}{\end{itemize}}
\newcommand{\bli}{\begin{list}{}{\labelwidth6mm\leftmargin8mm}}
\newcommand{\eli}{\end{list}}
\newcommand{\bpr}{\begin{proof}}
\newcommand{\epr}{\end{proof}}
\newcommand{\ds}{\displaystyle}

\newcommand{\R}{{\mathbb R}}
\newcommand{\real}{\R}
\newcommand{\Rn}{{\mathbb R}^{d}}
\newcommand{\rd}{\ensuremath{{\mathbb R}^\nd}}

\newcommand{\N}{\ensuremath{\mathbb N}}
\newcommand{\No}{\N_{0}}
\newcommand{\nat}{\N}
\newcommand{\no}{\No}
\newcommand{\Z}{\mathbb Z} 
\newcommand{\Zn}{\Z^{\nd}}
\newcommand{\zd}{\ensuremath{\mathbb{Z}^\nd}}
\newcommand{\C}{{\mathbb C}}
\newcommand{\ve}{\ensuremath\varepsilon}
\newcommand{\Gp}{{\mathcal G}_p}
\newcommand{\Gpx}[1]{{\mathcal G}_{#1}}
\newcommand{\Gpd}{{\mathcal G}_p(\mathbb{D})}

\newcommand{\ls}{\lesssim}

\newcommand{\dint}{\;\mathrm{d}}

\newcommand{\dist}{\ensuremath{\operatorname{dist}}}
 
\newcommand{\id}{\ensuremath{\operatorname{id}}}

\newcommand{\whole}[1]{\ensuremath\left\lfloor #1 \right\rfloor}
\newcommand{\Mu}{{\mathcal M}_{u,p}}
\newcommand{\Me}{{\mathcal M}_{u_1,p_1}}
\newcommand{\Mz}{{\mathcal M}_{u_2,p_2}}
\newcommand{\Mp}{\ensuremath{{\mathcal M}_{\varphi,p}}}

\newcommand{\mmbet}{\ensuremath{{m}^{2^{j\nd}}_{u_1,p_1}}}  
  
\newcommand{\mpb}{\ensuremath{m^{2^{j\nd}}_{\varphi,p}}} 
\newcommand{\mpbet}{\ensuremath{{m}^{2^{j\nd}}_{\varphi_1,p_1}}}  
\newcommand{\mpbzt}{\ensuremath{{m}^{2^{j\nd}}_{\varphi_2,p_2}}}  
\newcommand{\ms}{\ensuremath{{ m}_{u,p}}}
\newcommand{\mse}{\ensuremath{{ m}_{u_1,p_1}}}
\newcommand{\msz}{\ensuremath{{ m}_{u_2,p_2}}}
\newcommand{\mps}{\ensuremath{{ m}_{\varphi,p}}}
\newcommand{\mpse}{\ensuremath{{ m}_{\varphi_1,p_1}}}
\newcommand{\mpsz}{\ensuremath{{ m}_{\varphi_2,p_2}}}
\newcommand{\rphi}[1]{\ensuremath{\mathsf{r}_{#1}}}

\title{Generalised Morrey sequence spaces}
\author{Dorothee D. Haroske and Leszek Skrzypczak}

\begin{document}
\maketitle

\begin{abstract}
Generalised Morrey (function) spaces  enjoyed some  interest recently and found applications to PDE. Here we turn our attention to their discrete counterparts. We define  generalised Morrey sequence spaces $m_{\varphi,p}=m_{\varphi,p}(\mathbb{Z}^d)$. They are natural generalisations of the classical Morrey sequence  spaces $m_{u,p}$, $0<p\le u<\infty$,  which  were studied earlier.   
We consider some basic features   of the spaces as well as   embedding properties such as continuity, compactness and strict singularity.\\



\end{abstract}

\maketitle

\section{Introduction}\label{intro}
     Morrey  spaces and, in particular, generalised Morrey spaces  
	were studied in recent years quite intensively and systematically. Originally, Morrey spaces were introduced in \cite{Mor38}, when studying solutions of second order quasi-linear elliptic equations in the framework of Lebesgue spaces.  They can be understood as a complement of the Lebesgue spaces $L_p(\Rn)$.  Morrey spaces describe the local behavior of the $L_p$ norm, which makes them useful in  harmonic analysis,  potential analysis and PDE, e.g. when describing the local behavior of solutions of nonlinear partial differential equations, cf. e.g. \cite{KY94,Tay92}.  We refer to the recent  monograph 
	\cite{FHS-MS-2} for details and further  references.  Apart from the  Morrey spaces also the smoothness spaces of Besov and Triebel-Lizorkin type related to Morrey spaces  were introduced and  studied intensively, cf. e.g. \cite{Maz03b,SYY10,Tri13}. In particular,  decomposition methods like atomic or wavelet characterisations, that require suitably adapted sequence spaces, were used for the spaces, cf. \cite{Saw2,ST2}.     
	
	Also the discrete version of Morrey spaces attracted some attention.  To the best of our knowledge  the first paper devoted to Morrey sequence spaces was \cite{GuKiSc}. The authors study spaces $m_{u,p}(\Z)$ of the sequences indexed by integers. As  natural generalisation  we introduced in \cite{HaSk_Mo_seq} the Morrey sequence spaces $m_{u,p}(\mathbb{Z}^d)$, $0<p\le u<\infty$,  on $\mathbb{Z}^d$, $d\ge 1$, that is, we studied sequences indexed by the set of all lattice points in $\Rn$ having integer components.   The way how sequences are indexed matters since Morrey spaces, are not rearrangement invariant. The sequence spaces can be considered as a generalisation of  $\ell_p=\ell_p(\mathbb{Z}^d)$ spaces since $m_{p,p}(\mathbb{Z}^d) = \ell_p(\mathbb{Z}^d)$.  Geometrical features  of the  spaces and properties of operators acting on them were later studied by different authors, cf. \cite{GHI,GHIN,GKSS,Gu-Schwa,G-P,KS}.  In particular, in \cite{HaSk_Mo_seq} we describe the properties of embeddings of the spaces and prove the Morrey spaces version of the classical Pitt theorem. 
	
	The aim of this paper is to give the analogous discrete version of generalised Morrey function spaces $\mathcal{M}_{\varphi,p}(\R^d)$. The definition of these spaces goes back to \cite{Nak94}. We refer to \cite{NNS16}, \cite{Saw18} and the second volume of  \cite{FHS-MS-2}. The wavelet characterisation of the corresponding smoothness spaces were proved in \cite{HMS} and \cite{HLMS}. The generalised Morrey sequence spaces $m_{\varphi,p}(\Zn)$ are introduced in Section~\ref{sec-1}. If the function $\varphi(t)\sim t^{d/u}$, $u\geq p$, $t>0$, then the space  $m_{\varphi,p}(\Zn)$ coincides with $m_{u,p}(\Zn)$.  The basic properties of the sequence spaces are proved in this section. Section~\ref{sec-2} is devoted to the properties of embeddings of the spaces. We find the sufficient and necessary conditions for the continuity of the embeddings. In Section~\ref{sec-3} the norms of the embeddings of finite-dimensional generalised Morrey sequence spaces are calculated. The embeddings of the infinite-dimensional spaces are never compact, therefore in the concluding Section~\ref{sec-4} we study the weaker operator property, namely the strict singularity. Once more the sufficient and necessary conditions are proved.

\section{Morrey sequence spaces}\label{sec-1}
\subsection{Preliminaries}

First we fix some notation. By $\nat$ we denote the \emph{set of natural numbers}, by $\no$ the set $\nat \cup \{0\}$,  and by $\Zn$ the \emph{set of all lattice points in $\Rn$ having integer components}.
For $a\in\real$, let   $\whole{a}:=\max\{k\in\Z: k\leq a\}$ and $a_+:=\max\{a,0\}$.
All unimportant positive constants will be denoted by $C$, occasionally with
subscripts. By the notation $A \lesssim B$, we mean that there exists a positive constant $C$ such that
$A \le C \,B$, whereas  the symbol $A \sim B$ stands for $A \ls B \ls A$.
We denote by $|\cdot|$ the Lebesgue measure when applied to measurable subsets of $\Rn$. Given two (quasi-)Banach spaces $X$ and $Y$, we write $X\hookrightarrow Y$
if $X\subset Y$ and the natural embedding of $X$ into $Y$ is continuous.

For each cube $Q\subset\rd$, we denote by $\ell(Q)$ the side length of $Q$ and  $|Q|=\ell(Q)^d$ its volume.  For $x\in\rd$ and $r\in(0,\infty)$, we denote by $Q(x,r)$ the compact cube centred at $x$ with side length $r$, whose sides are parallel to the axes of coordinates. Let $\mathcal{Q}$ denote the set of all dyadic cubes in $\rd$, namely, $\mathcal{Q}:=\{Q_{j,k}:=2^{-j}([0,1)^d+k):j\in\mathbb{Z}, k\in\zd\}$.

For $0<p<\infty$ we denote by $\ell_p=\ell_p(\Zn)$, 
\begin{align*} 
\ell_p(\Zn) = \Big\{\lambda=  \{\lambda_k\}_{k\in \Zn}\subset \C: \|\lambda| \ell_p \| = \Big(\sum_{k\in\Zn} |\lambda_k|^p\Big)^\frac{1}{p} < \infty \Big\}, 
\end{align*} 
complemented by
\begin{align*} 
\ell_\infty(\Zn) = \Big\{\lambda=  \{\lambda_k\}_{k\in \Zn}\subset \C: \|\lambda| \ell_\infty \| = \sup_{k\in\Zn} |\lambda_k| < \infty \Big\}. 
\end{align*}
Finally, we adopt the custom to denote by $c=c(\Zn)$, $c_0=c_0(\Zn)$, and $c_{00}=c_{00}(\Zn)$ the corresponding subspaces of $\ell_\infty(\Zn)$ of convergent, null and finite sequences, respectively, that is
\begin{align*}
c & = \{ \lambda= \{\lambda_k\}_{k\in \Zn}\in\ell_\infty: \ \exists\ \mu\in\C: |\lambda_k-\mu|\xrightarrow[|k|\to\infty]{} 0\Big\},\\
c_0 & = \{ \lambda= \{\lambda_k\}_{k\in \Zn}\in\ell_\infty: \ |\lambda_k|\xrightarrow[|k|\to\infty]{} 0\Big\},\\
c_{00} & = \{ \lambda= \{\lambda_k\}_{k\in \Zn}\in\ell_\infty: \ \exists\ r_0\in\no: \lambda_k=0\quad\text{for}\quad |k|>r_0\Big\}.
\end{align*}

For later use, let us also recall the definition of Lorentz sequence spaces $\ell_{p,q}=\ell_{p,q}(\Zn)$, $0<p<\infty$, $0<q\leq\infty$, which contain all sequences $\lambda=\{\lambda_k\}_{k\in\Zn}\in c_0$, such that for $0<q<\infty$, 
\begin{align}\label{lorentz-1} 
  \ell_{p,q}(\Zn) = \Big\{\lambda=  \{\lambda_k\}_{k\in \Zn}\subset \C: \|\lambda| \ell_{p,q} \| = \Big(\sum_{\nu\in\nat} \nu^{\frac{q}{p}-1} (\lambda^*_\nu)^q\Big)^{\frac1q} < \infty \Big\}, 
\end{align}
complemented by 
\begin{align}\label{lorentz-2} 
\ell_{p,\infty}(\Zn) = \Big\{\lambda=  \{\lambda_k\}_{k\in \Zn}\subset \C: \|\lambda| \ell_{p,\infty} \| = \sup_{\nu\in\nat} \nu^{1/p} \lambda^*_\nu < \infty \Big\}. 
\end{align}
Here the non-increasing rearrangement $\{\lambda^*_\nu\}_{\nu\in\nat}$ of a null sequence $\lambda=\{\lambda_k\}_{k\in\Zn}$ is defined as follows:
\begin{equation*}
  \lambda^*_\nu = \inf\{t\geq 0: \mu(\lambda,t)\leq \nu-1\} \quad\text{with}\quad \mu(\lambda,t) = \#\{m\in\Zn: |\lambda_m|>t\}, 
\end{equation*}
for $\nu\in\nat$. Then $\lambda_1^* \geq \lambda_2^*\geq \dots\geq 0$ and $\lambda_1^*=\|\lambda|\ell_\infty\|$. The spaces $\ell_{p,q}$ are well-known to be (quasi-)Banach spaces, and satisfy $\ell_p = \ell_{p,p}$ and $\ell_{p,q}\hra \ell_{p,r}$ when $r\geq q$, cf. \cite{BR,CS,LT1}.\\

As we mostly deal with sequence spaces on $\Zn$ we shall often omit it from their notation, for convenience.

\subsection{The background: Morrey function spaces}
In this paper we consider generalised Morrey sequence spaces $\mps$ which on the one hand generalise the Morrey function spaces $\ms$ studied in \cite{HaSk_Mo_seq} in some detail, and which on the other hand can be seen as the natural counterpart to the Morrey function spaces $\Mp$ which are generalisations of the (classical) Morrey spaces $\Mu$ themselves. In both cases the parameter $u$ is replaced by a function $\varphi$ according to the following definition. 

Recall first  that the classical  Morrey space
${\mathcal M}_{u,p}(\rd)$, $0<p\leq u<\infty $, is defined to be the set of all
  locally $p$-integrable functions $f\in L_p^{\mathrm{loc}}(\rd)$  such that
$$
\|f \mid {\mathcal M}_{u,p}(\rd)\| :=\, \sup_{x\in\rd, r>0} |Q(x,r)|^{\frac{1}{u}}
\left( {\frac{1}{|Q(x,r)|}}\int_{Q(x,r)} |f(y)|^p \dint y \right)^{\frac{1}{p}}\, <\, \infty\, .
$$

\begin{definition}\label{def-Mp}
Let $0<p<\infty$ and $\varphi:(0,\infty)\rightarrow [0,\infty)$ be a function which does not satisfy $\varphi\equiv 0$. Then $\Mp(\rd)$ is the set of all 
locally $p$-integrable functions $f\in L_p^{\mathrm{loc}}(\rd)$ for which 
\begin{equation} \label{Morrey-norm}
\|f \mid \Mp(\rd)\| :=\, \sup_{x\in\rd, r>0} \varphi (r)
\left( {\frac{1}{|Q(x,r)|}}\int_{Q(x,r)} |f(y)|^p \dint y \right)^{\frac{1}{p}}\, <\, \infty\, .
\end{equation}
\end{definition}

\begin{remark} \label{rmk1}
The above definition goes back to \cite{Nak94}.
When $\varphi(t)=t^{\frac{\nd}{u}}$ for $t>0$ and $0<p\leq u<\infty$, then $\Mp(\rd)$ coincides with $\Mu(\rd)$, which in turn recovers the Lebesgue space $L_p(\rd)$ when $u=p$. 
In the definition of $\|\cdot\mid \Mp(\rd)\|$ balls or cubes with sides parallel to the axes of coordinates can be taken. This  change leads to  equivalent quasi-norms. Note that for $\varphi_0\equiv 1$ (which would correspond to $u=\infty$) we obtain 
\begin{equation}\label{M1p}
{\mathcal M}_{\varphi_0,p}(\rd) = L_\infty(\rd),\quad 0<p<\infty,\quad \varphi_0\equiv 1,
\end{equation}
due to Lebesgue's differentiation theorem.  

When $\varphi(t)=t^{-\sigma} \chi_{(0,1)}(t)$ where  $-\frac{\nd}{p}\leq \sigma <0$, then $\Mp(\rd)$ coincides with the local Morrey spaces $\mathcal{L}^{\sigma}_p(\rd)$ introduced by Triebel in \cite[Section~1.3.4]{Tri13}. If $ \sigma=-\frac{\nd}{p}$, then the space is a uniform Lebesgue space $\mathcal{L}_p(\rd)$.
\end{remark}

\bigskip 

For $\Mp(\rd)$ it is usually required that $\varphi\in \Gp$, where $ \Gp$ 
is the set of all nondecreasing functions $\varphi:(0,\infty)\rightarrow (0,\infty)$ such that $\varphi(t)t^{-\nd/p}$ is a nonincreasing function, i.e.,
\begin{equation}\label{Gp-def}
1\leq \frac{\varphi(r)}{\varphi(t)}\leq \left(\frac{r}{t}\right)^{\nd/p}, \quad 0<t\leq r<\infty.
\end{equation}
A justification for the use of the class $\Gp$ comes from Lemma~\ref{M-not-triv} below, cf. e.g. \cite[Lem. 2.2]{NNS16}.  One can easily check that $\mathcal{G}_{p_2}\subset \mathcal{G}_{p_1}$ if $0<p_1\le p_2<\infty$.

 \begin{remark}
   Note that $\varphi\in\Gp$ enjoys a doubling property, i.e., $\varphi(r)\leq \varphi(2r)\leq 2^\frac{d}{p}\varphi(r)$, $0<r<\infty$, {therefore} we can define an equivalent quasi-norm in $\Mp(\rd)$ by taking the supremum over the collection of all dyadic cubes, namely,
    \begin{align}\label{Mf-norm}
        \|f \mid \Mp(\rd)\| :=\sup_{P\in \mathcal{Q}} \varphi(\ell(P))
             \left(\frac{1}{|P|}\int_{P} |f(y)|^p \dint y \right)^{\frac{1}{p}}.
    \end{align}
    In the sequel, we always assume that $\varphi\in\Gp$ and consider the quasi-norm \eqref{Mf-norm}.   This choice is natural and covers all interesting cases. Moreover, we shall usually assume in the sequel that $\varphi(1)=1$. 
This simplifies the notation, but does not limit the generality of considerations. Note that for a function $\varphi\in\Gp$ with $\varphi(1)=1$, then \eqref{Gp-def} implies that
\begin{align}\label{Gp-tlarge}
& 1\leq \varphi(t)\leq t^{\frac{\nd}{p}},\quad t\geq 1, \quad \text{in particular},\quad 
  \varphi(t) t^{-\frac{\nd}{p}} \leq 1\quad \text{for}\ t\to\infty,
  \intertext{and}\label{Gp-tsmall}
&  t^{\frac{\nd}{p}}\leq \varphi(t)\leq 1,\quad 0<t\leq 1.
\end{align}   
Together with the required monotonicity conditions for $\varphi\in\Gp$  this implies the existence of both limits
\begin{equation}\label{lim-exist}
  \lim_{t\to\infty} \varphi(t) t^{-\frac{\nd}{p}} \in [0,1] \qquad \text{and}\qquad \lim_{t\to 0^+} \varphi(t) \in [0,1].
  \end{equation}
Later it will turn out that the special cases when one or both limits equal $0$ play an important r\^ole. 
\end{remark}

\begin{lemma}[\cite{NNS16,Saw18}]\label{M-not-triv}
  Let $0<p<\infty$ and $\varphi:(0,\infty)\rightarrow [0,\infty)$ be a function satisfying $\varphi(t_0)\neq 0$ for some $t_0>0$.
  \begin{enumerate}[\upshape\bfseries (i)]
    \item
Then $\Mp(\rd)\neq \{0\}$ if, and only if, 
$$
\sup_{t>0} \varphi(t)  \min (t^{-\frac{\nd}{p}},1) < \infty.
$$
\item Assume \ $\displaystyle\sup_{t>0} \varphi(t)  \min (t^{-\frac{\nd}{p}},1) < \infty$. Then there exists $\varphi^*\in\Gp$ such that
$$
\Mp(\rd)={\mathcal M}_{\varphi^*,p}(\rd)
$$
in the sense of equivalent (quasi-)norms.
 \end{enumerate}
\end{lemma}

  \begin{remark}\label{R-emb-Mphi}
    In \cite[Thm.~3.3]{GHLM17} it is shown that for $1\leq p_2\leq p_1<\infty$, $\varphi_i\in \mathcal{G}_{p_i}$, $i=1,2$, then
    \begin{equation} \label{emb-Mphi}
      {\mathcal M}_{\varphi_1, p_1}(\rd) \hookrightarrow {\mathcal M}_{\varphi_2, p_2}(\rd)
      \end{equation}
    if, and only if, there exists some $C>0$ such that for all $t>0$, $\varphi_1(t)\geq C \varphi_2(t)$. 
   The argument can be immediately extended to $0<p_2\leq p_1<\infty$, cf. \cite[Cor.~30]{FHS-MS-2}.
   
 In case of $\varphi_i(t) =t^{\nd/u_i}$, $0<p_i\leq u_i<\infty$, $i=1,2$, it is well-known that
\begin{equation}\label{emb-Mu}
\Me(\rd) \hookrightarrow \Mz(\rd)\qquad\text{if, and only if,}\qquad
p_2\leq p_1\leq u_1=u_2, 
\end{equation}
cf. \cite{piccinini-1} and \cite{rosenthal}. This can be observed from \eqref{emb-Mphi} { since} 
\[
\varphi_1(t)\geq C \varphi_2(t) \quad\text{{means}}\quad t^{\frac{\nd}{u_1}} \geq C t^{\frac{\nd}{u_2}},\quad t>0,
\]
which results in $u_1=u_2$.

Let $p\in (0,\infty)$, $\varphi\in \Gp$, and $0<r\leq p$. Then \eqref{emb-Mphi} applied to $\varphi_1=\varphi$, $\varphi_2(t)=t^{\nd/r}$ results in
\begin{equation}
  \Mp(\rd)\hookrightarrow L_r(\rd) \quad\text{if, and only if,}\quad \varphi(t) \geq C t^{\frac{\nd}{r}},\quad t>0.
  \end{equation}
In connection with \eqref{Gp-tlarge} and \eqref{Gp-tsmall} this leads to $r=p$, that is,
\begin{equation}\label{Mphi-Lr}
 \Mp(\rd)\hookrightarrow L_r(\rd) \quad\text{if, and only if,}\quad r=p \quad \text{and}\quad \varphi(t) \sim t^{\nd/p},\quad t>1.
\end{equation}

Finally, if there exists some $u\geq p$ such that $\varphi(t) \leq c t^{\nd/u}$, $t>0$, then
\begin{equation}\label{emb-Mu-Mphi}
  \mathcal{M}_{u,r}(\rd)\hookrightarrow \Mp(\rd)\quad \text{for all $r$ with $p\leq r\leq u$}.
\end{equation}

  \end{remark}

 \begin{remark}\label{emb-Linf}
   Recall Sawano's observation, cf. \cite{Saw18},  that {if} $\ \inf_{t>0} \varphi(t)>0$, then $\mathcal{M}_{\varphi,p}(\rd)\hookrightarrow L_\infty(\rd)$, while 
$ \sup_{t>0} \varphi(t)<\infty $ implies $ L_\infty(\rd)\hookrightarrow\mathcal{M}_{\varphi,p}(\rd)$, leading as a special case to \eqref{M1p}.
 In particular, when $\varphi\in\Gp$, $0<p<\infty$, then according to \cite[Thms. 27, 28]{FHS-MS-2},
     \begin{equation}\label{emb-Linf-sharp}
       \Mp(\rd) \hookrightarrow L_\infty(\rd) \quad\text{if, and only if,}\quad 
       \inf_{t>0} \varphi(t)>0,
  \end{equation}
and
     \begin{equation}\label{Linf-emb-sharp}
       L_\infty(\rd) \hookrightarrow \Mp(\rd) \quad\text{if, and only if,}\quad 
       \sup_{t>0} \varphi(t)<\infty.
  \end{equation}
  \end{remark}

We consider the following examples.

\begin{examples}\label{exm-Gp}
\begin{enumerate}[\upshape\bfseries (i)]
\item The function 
\begin{equation}\label{example1} \varphi_{u,v}(t)=
\begin{cases}
t^{\nd/u} & \text{if}\qquad t\le 1,\\ 
t^{\nd/v} & \text{if}\qquad t >1, 
\end{cases}
\end{equation}
with $0<u,v<\infty$ belongs to $\Gp$ with $p=\min(u,v)$. In particular, taking $u=v$, the
function $\varphi(t)=t^{\frac{\nd}{u}}$ belongs to $\mathcal{G}_p$ whenever  $0<p\leq u<\infty$. 
\item The function  $\varphi(t)=\max(t^{\nd/v},1)$ belongs to $\mathcal{G}_v$. It corresponds to \eqref{example1} with $u=\infty$.
\item The function  $\varphi(t)=\sup\{s^{\nd/u}\chi_{(0,1)}(s): s\le t\}= \min(t^{\nd/u},1)$ defines an equivalent (quasi)-norm in $\mathcal{L}^{\sigma}_p(\rd)$, $\sigma= -\frac{\nd}{u}$, $p\le u$. The function $\varphi$  belongs to $\mathcal{G}_u\subset \mathcal{G}_p$. It corresponds to \eqref{example1} with $v=\infty$. 
\item  The   function $\varphi(t)=t^{\nd/u}(\log (L+t))^a$, with $L$ being a sufficiently large constant, belongs to $\mathcal{G}_u$ if  $0< u<\infty$ and $a\leq 0$. 
\end{enumerate}
\smallskip
Other examples   can be found, for instance, in \cite[Ex.~3.15]{Saw18}.
\end{examples}

\subsection{Morrey sequence spaces: definition and basic properties}
In \cite{GuKiSc} different versions of Morrey sequence spaces were introduced, essentially on ${\N}$ or $\Z$, respectively, including the counterparts $\ms$ and $\mps$ of the Morrey function spaces $\Mu(\Rn)$ and $\Mp(\Rn)$, respectively. The authors in \cite{GuKiSc} restrict their attention to the Banach case, i.e., $1\leq p\leq u<\infty$, and $\varphi\in\Gp$. In our paper \cite{HaSk_Mo_seq} we extended the definition of $\ms$ to the general $\nd$-dimensional case and $0<p\leq\infty$, and studied embeddings, duality questions and further properties in greater detail. This is our intention now as well. Let us first recall the definition of the spaces $\ms$, cf. \cite{HaSk_Mo_seq}. Let $0<p\le u < \infty$. Then $\ms=\ms(\Zn)$ is defined by
	\begin{align} 
	\ms(\Zn) = & \Big\{\lambda= \{\lambda_k\}_{k\in \Zn}\subset \C:  \nonumber\\
	&              \|\lambda| \ms \| = \sup_{j\in\no; m\in \Zn}\frac{|Q_{-j,m}|^{1/u}}{ |Q_{-j,m}|^{1/p} } \Big(\sum_{k:\, Q_{0,k}\subset Q_{-j,m}} \!\!\!|\lambda_k|^p\Big)^\frac{1}{p} < \infty \Big\}. 
	\end{align} 

Parallel to the generalisation of $\Mu(\Rn)$ to $\Mp(\Rn)$ we now introduce the modified version of $\ms(\Zn)$. In view of \eqref{Mf-norm} only values of $\varphi$ in points $2^j$, $j\in\No$, are of interest now. Let $\mathbb{D}=\{2^j: j\in \N_0\}$.
        
\begin{definition} \label{D-mps}
Let $0<p<\infty$ and $\varphi: {\mathbb{D}} \rightarrow [0,\infty)$ be a function which does not satisfy $\varphi\equiv 0$. We define $\mps=\mps(\Zn)$ by
	\begin{align} 
	\mps(\Zn) = & \Big\{\lambda=  \{\lambda_k\}_{k\in \Zn}\subset \C:  \nonumber\\
	&              \|\lambda| \mps \| = \sup_{j\in\no; m\in \Zn} \frac{\varphi(\ell(Q_{-j,m}))}{ |Q_{-j,m}|^{1/p} } \Big( \sum_{k:\, Q_{0,k}\subset Q_{-j,m}} \!\!\!|\lambda_k|^p\Big)^\frac{1}{p} < \infty \Big\}. 
	\end{align} 
\end{definition}

\begin{remark}
  In \cite{GuKiSc} the corresponding one-dimensional counterpart was introduced and studied. Obviously, with $\varphi(t)= t^{d/u}$, recall Example~\ref{exm-Gp}(i), we regain $\ms(\Zn)$.
  
Moreover, Definition~\ref{D-mps} gives the discrete counterpart of $\Mp(\Rn)$ in view of \eqref{Morrey-norm}. More precisely, given some sequence $\lambda=\{\lambda_k\}_{k\in \Zn} \in \mps$, then
		\[
		f_\lambda = \sum_{k\in\Zn} \lambda_k \chi_{Q_{0,k}} \in \Mp,
		\]
		where $\chi_A$ denotes the characteristic function of a set $A\subset\Rn$, as usual. 
		
		Conversely, let $f_0\in\Mp$ and define $\lambda^0 = \{\lambda_k^0\}_{k\in\Zn}$ by
		\[\lambda_k^0 = \Big(\int_{Q_{0,k}} |f_0(y)|^p\dint y\Big)^{1/p}, \quad k\in\Zn.\] Then  $\|\lambda^0|\mps\| \leq \|f_0|\Mp\|$. This can be verified immediately. Hence there is some obvious correspondence between the generalised Morrey sequence and function spaces, respectively.

\end{remark}

We shall list a number of more or less elementary properties, and then investigate, which properties of $\varphi$ imply  the spaces $\mps$ to be non-trivial. This can be seen as the counterpart of Lemma~\ref{M-not-triv}.

As a preparation we study the $\mps$-norm of a `characteristic' sequence. This extends \cite[Lemma~4.2]{GuKiSc}.

\begin{lemma}\label{char-seq}
  Let $k_0\in\no$ and $m_0\in\Zn$ and consider the sequence $\lambda^0 = \{\lambda^0_m\}_{m\in\Zn}$ with
  \[
    \lambda_m^0 = \begin{cases} 1, & \text{if}\quad Q_{0,m} \subset Q_{-k_0, m_0}, \\ 0, & \text{otherwise}.\end{cases} 
  \]
  Assume that
  \begin{equation}\label{ddh-4}
\sup_{j\geq k_0} \varphi(2^j) 2^{(k_0-j)\frac{\nd}{p}} < \infty.
    \end{equation}
  Then $\lambda^0 \in \mps$ with
  \begin{equation}\label{norm-l0}
    \|\lambda^0 | \mps\| = \sup_{j\in \no}  \varphi(2^j) \min(2^{(k_0-j)\frac{\nd}{p}},1).
    \end{equation} 
  \end{lemma}

  \begin{proof}
    In view of Definition~\ref{D-mps} we need to show that
      \begin{align*}
\|\lambda^0| \mps \| = \sup_{j\in\no} \varphi(2^j) 2^{-j\frac{\nd}{p}} \sup_{ m\in\Zn} |Q_{-j,m}\cap Q_{-k_0,m_0}|^{\frac1p} < \infty.
      \end{align*}
      Let us denote by
      \[
        \varphi^\ast(2^{k_0})=  \sup_{j\geq k_0} 2^{(k_0-j)\frac{\nd}{p}} \varphi(2^j), \qquad \varphi^+(2^{k_0}) = \sup_{j\leq k_0} \varphi(2^j).
      \]
By assumption \eqref{ddh-4}, $\varphi^\ast(2^{k_0})<\infty$, while $\varphi^+(2^{k_0}) = \max_{j=0, \dots, k_0} \varphi(2^j) = c_{k_0}$ is always finite. Thus for \eqref{norm-l0} it is sufficient to show that 
\begin{equation}\label{norm-l0-a}
  \|\lambda^0| \mps \| = \max\{ \varphi^\ast(2^{k_0}),\varphi^+(2^{k_0})\}.    
\end{equation}
      If $j\geq k_0$, then
      \begin{align*}
        \sup_{j\geq k_0} \varphi(2^j) 2^{-j\frac{\nd}{p}} \sup_{m\in\Zn} |Q_{-j,m}\cap Q_{-k_0,m_0}|^{\frac1p} & = \sup_{j\geq k_0} \varphi(2^j) 2^{-j\frac{\nd}{p}} |Q_{-k_0,m_0}|^{\frac1p} \\
        & = \sup_{j\geq k_0} \varphi(2^j) 2^{-(j-k_0)\frac{\nd}{p}} = \varphi^\ast(2^{k_0}),
        \end{align*}
        while for $j\leq k_0$,
        \begin{align*}
          \sup_{j\leq k_0} \varphi(2^j) 2^{-j\frac{\nd}{p}} \sup_{m\in\Zn} |Q_{-j,m}\cap Q_{-k_0,m_0}|^{\frac1p} & = \sup_{j\leq k_0} \varphi(2^j) 2^{-j\frac{\nd}{p}} 2^{j\frac{\nd}{p}}   = 
                                                                                                                   \varphi^+(2^{k_0}).
        \end{align*}
        This concludes the argument.
\end{proof}

Next we give the discrete counterpart of Lemma~\ref{M-not-triv}(i).
\begin{lemma}\label{m-not-triv}
	Let $0<p<\infty$ and $\varphi: {\mathbb{D}}\rightarrow [0,\infty)$ be a function satisfying $\varphi(2^{k_1})\neq 0$ for some $k_1\in\no$. 
	Then $\mps(\Zn)\neq \{0\}$ if, and only if, 
	\begin{equation} \label{ddh-3}
		\sup_{j\in\no} \varphi(2^j)  2^{-j\frac{\nd}{p}} < \infty.
	\end{equation}
\end{lemma}

\begin{proof} We adapt the arguments in the proof of Lemma~\ref{M-not-triv}, cf. \cite{NNS16,Saw18}. 
	
	Assume first \eqref{ddh-3} to hold. Choose $k_0=0$ and $m_0\in\Zn$ arbitrary, then by Lemma~\ref{char-seq}, $\lambda^0\in\mps$ with
	\[
	0<\varphi(2^{k_1}) 2^{-k_1\frac{\nd}{p}} \leq \|\lambda^0| \mps\|  = \sup_{j\in\no} \varphi(2^j)  2^{-j\frac{\nd}{p}} < \infty
	\]
	by assumption. Thus $\mps\supsetneq \{0\}$.
	
	Now assume $\mps\neq\{0\}$, hence there exists some $\lambda\in\mps\setminus\{0\}$. This implies that there exists $j_0=j(\lambda)\in\no$ and $m_0=m_0(\lambda)\in\Zn$ such that
	\[
	\sum_{k: Q_{0,k}\subset Q_{-j_0, m_0}} |\lambda_k|^p = c_\lambda^p >0.
	\]
	We want to show \eqref{ddh-3} and proceed by contradiction: assume that there exists some strictly increasing sequence $\{j_l\}_{l\in\no}$ such that
	\[ \varphi(2^{j_l}) 2^{-j_l \frac{\nd}{p}} > 2^l,\quad l\in\nat. \]
	Then
	\begin{align*}
		\|\lambda|\mps\| & = \sup_{j\in\no, m\in\Zn} \varphi(2^j) 2^{-j\frac{\nd}{p}} \Big(\sum_{k: Q_{0,k}\subset Q_{-j, m}} |\lambda_k|^p\Big)^{\frac1p}\\
&		\geq \sup_{j\geq j_0} \varphi(2^j) 2^{-j\frac{\nd}{p}} \Big(\sum_{k: Q_{0,k}\subset Q_{-j_0, m_0}} |\lambda_k|^p\Big)^{\frac1p}\\
		&> c_\lambda\ \sup_{l\geq l_0} \varphi(2^{j_l})  2^{-j_l\frac{\nd}{p}} \ > c_\lambda \ \sup_{l\geq l_0} 2^l = \infty,
	\end{align*}
	which contradicts $\lambda\in\mps$ and concludes the argument.
\end{proof}


\begin{proposition}\label{Prop-1-mps}
  Let $0<p<\infty$ and $\varphi: {\mathbb{D}}\rightarrow [0,\infty)$ be a function satisfying $\varphi(2^{k_0})\neq 0$ for some $k_0\in\no$. 
\begin{enumerate}[\upshape\bfseries (i)]
	\item
	$\mps$ is a (quasi-) Banach space. 
	\item
	If $0<p_2\leq p_1<\infty$, then $m_{\varphi,p_1} \hookrightarrow m_{\varphi,p_2}$.
	\item
Let $\varphi_1, \varphi_2: {\mathbb{D}}\rightarrow [0,\infty)$. If  {there exists some $C>0$ such that $\varphi_1(2^j)\geq C \varphi_2(2^j)$, for any $j\in\no$,} then  $m_{\varphi_1,p} \hookrightarrow m_{\varphi_2,p}$.
	\item
	We have $\mps\hookrightarrow \ell_\infty$.
      \item
        If $\varphi$ is also bounded from above, that is, $\sup_{k\in\no}\varphi(2^k)<\infty$, then $\mps = \ell_\infty$.
      \end{enumerate}
\end{proposition}

\begin{proof}
Part (i) is standard, the completeness can be shown similar to the (one-dimensional) counterpart in \cite{GuKiSc}. 
Assertion (ii)  follows from the monotonicity of $\ell_p$ spaces in $p$ and H\"older's inequality, while (iii) is obvious by Definition~\ref{D-mps}.

Now we prove (iv). Let $m\in\Zn$. For convenience, let us assume $k_0=0$ and  $\varphi(1)=\varphi(\ell(Q_{0,m})))=|Q_{0,m}|=1$. Then
	\begin{align*}
	|\lambda_m|= & \varphi(\ell(Q_{0,m}))  \left(|Q_{0,m}|^{-1}\sum_{k:Q_{0,k}\subset Q_{0,m}} \!\!\!|\lambda_k|^p\right)^{\frac1p}\\
	\leq & \sup_{j\in\no; m\in \Zn} \varphi(\ell(Q_{-j,m}))  \Big(|Q_{-j,m}|^{-1}\sum_{k:\, Q_{0,k}\subset Q_{-j,m}}\!\!\!  |\lambda_k|^p\Big)^\frac{1}{p}
	=  \|\lambda|\mps\|,
	\end{align*}
	such that finally, taking the supremum over all $m\in\Zn$,
	\[\|\lambda|\ell_\infty\| \leq \|\lambda|\mps\| .\]
 So it remains to deal with (v). By (iv) we only need to verify that $\ell_\infty \hookrightarrow \mps$ when $\sup_{k\in\no} \varphi(2^k) = c<\infty$. Let $\lambda \in \ell_\infty$. Thus for any $j\in\no$ and $m\in\Zn$,
	\[
\frac{\varphi(\ell(Q_{-j,m}))}{|Q_{-j,m}|^{1/p}} 
	\Big(\sum_{k:Q_{0,k}\subset Q_{-j,m}}\!\!\!\! |\lambda_k|^p\Big)^\frac{1}{p} \leq 
	c\ |Q_{-j,m}|^{-\frac{1}{p}} \|\lambda|\ell_\infty\|\ |Q_{-j,m}|^{\frac1p} 
 = c \|\lambda|\ell_\infty\|,
	\]
	which finally results in $\|\lambda|\mps\|\leq c\ \|\lambda|\ell_\infty\|$.
      \end{proof}

\begin{remark}
  Note that (v) is the counterpart of $\Mp(\Rn)=L_\infty(\Rn)$ (in the sense of equivalent norms), if, and only if, $0< \inf_{t>0} \varphi(t) \leq \sup_{t>0} \varphi(t)<\infty$, recall \eqref{emb-Linf-sharp}, \eqref{Linf-emb-sharp}. It can also be seen as the generalisation of $\ms(\Zn)$ in case of $u=\infty$. In Corollary~\ref{linf-mps} below we can prove that (v) is in fact an equivalence, if we assume $\varphi\in\Gp$.
\end{remark}
 
Note that the definition of $\varphi\in\Gp$ in \eqref{Gp-def} directly extends to our discrete case now, that is, we assume the sequence $\{\varphi(2^k)\}_{k\in\no}$ to be nondecreasing and $\{2^{-k\frac{\nd}{p}} \varphi(2^k)\}_{k\in\no}$ nonincreasing, i.e.,
\begin{equation}\label{Gp-def-disc}
1\leq \frac{\varphi(2^k)}{\varphi(2^j)}\leq 2^{(k-j)\frac{\nd}{p}}, \quad j\leq k, \quad j,k\in\no.
\end{equation}
{If the function   $\varphi:(0,\infty)\rightarrow (0,\infty)$ belongs to $\Gp$ in the sense of \eqref{Gp-def}, then the sequence  $\{\varphi(2^k)\}_{k\in\no}$ is nondecreasing and the estimates \eqref{Gp-def-disc} hold. Vice versa, 
any nondecreasing sequence  satisfying \eqref{Gp-def-disc}  is a restriction to ${\mathbb D}$ of some function belonging to $\Gp$. To prove  the last statement we can first define the function $\psi:(0,\infty)\rightarrow (0,\infty)$  being a step function taking the value $\varphi(2^k)$ on the segment $[2^k, 2^{k+1})$, $k\in\No$,  and  taking $\psi(t)=\varphi(1)$ if $0<t\le 1$. Afterwards we put 
\[ \phi(t)= t^\frac{d}{p}\sup_{s\ge t} \psi(s) s^{-d/p}, \quad t\in (0,\infty).\]  
 One can easily check using \eqref{Gp-def-disc} that $\phi(2^k)= \varphi(2^k)$ for any $k\in \N_0$. Moreover $\phi$ belongs to the $\Gp$ class. To prove this we can use an argument similar to the one used in Step 2 of the proof of Lemma~\ref{m-equiv-Gp} below. 
 To make the distinction between the discrete and the non-discrete case we shall write $\Gpd$ and $\Gp(\R_+)$ in the sequel, respectively. } 
  
\begin{example}\label{ex-rev}
  We return to Lemma~\ref{char-seq} and consider the sequence $\lambda^0 = \{\lambda^0_m\}_{m\in\Zn}$ again, defined by 
\[     \lambda_m^0 = \begin{cases} 1, & \text{if}\quad Q_{0,m} \subset Q_{-k_0, m_0}, \\ 0, & \text{otherwise},\end{cases} 
\]
where $k_0\in\no$ and $m_0\in\Zn$ are arbitrary. Observe that for $\varphi\in\Gpd$, $\varphi^+(2^{k_0}) = \varphi(2^{k_0}) = \varphi^\ast(2^{k_0})$, such that Lemma~\ref{char-seq}, in particular \eqref{norm-l0-a}, can be rewritten as 
\begin{equation} \|\lambda^0 | \mps\| = \varphi(\ell(Q_{-k_0, m_0})) = \varphi(2^{k_0}),\quad \varphi\in \Gpd.
    \end{equation}
\end{example}

We return to Prop.~\ref{Prop-1-mps}(v), now with the additional assumption {$\varphi\in\Gpd$}.

\begin{corollary}\label{linf-mps}
  Let $0<p<\infty$ and {$\varphi\in\Gpd$}. Then the following conditions are equivalent,
  \begin{enumerate}[\upshape\bfseries (i)]
\item  $\quad \sup_{k\in\no}\varphi(2^k)<\infty$,
\item
$ \quad \ell_\infty = \mps$, 
\item $\quad  \ell_\infty \hookrightarrow \mps$.
\end{enumerate}
  \end{corollary}

  \begin{proof}
    In view of Prop.~\ref{Prop-1-mps} it remains to show that $\ell_\infty\hookrightarrow \mps$ implies $\sup_{k\in\no}\varphi(2^k)<\infty$, where we now benefit from our additional assumption that {$\varphi\in\Gpd$}. In view of our above Example~\ref{ex-rev} we can argue, that for all $k_0\in\no$, 
    \[
\varphi(2^{k_0}) = \|\lambda^0 | \mps\| \leq c\ \|\lambda^0| \ell_\infty\| = c,
    \]
where we used $\ell_\infty\hookrightarrow \mps$. This concludes the argument. 
    \end{proof}
  
    \begin{remark}
This is obviously the counterpart of \eqref{Linf-emb-sharp}. Moreover, it explains why we shall often exclude the case $\sup_{k\in\no} \varphi(2^k)<\infty$ from our considerations below, as the space $\ell_\infty$ is already well-studied, unlike $\mps$.
    \end{remark}

\begin{lemma}\label{m-equiv-Gp}
Let $0<p<\infty$ and $\varphi: \N\rightarrow [0,\infty)$ be a function satisfying $\varphi(2^{k_0})\neq 0$ for some $k_0\in\no$. 
Assume that $\displaystyle\sup_{k\in\no} \varphi(2^k)  2^{-k\frac{\nd}{p}} < \infty$. Then there exists some {$\tilde{\varphi}\in\Gpd$} such that
$$
\mps(\Zn)= m_{\tilde{\varphi},p}(\Zn)
$$
in the sense of equivalent (quasi-)norms.
\end{lemma}

\begin{proof}~
In view of our assumptions, Lemma~\ref{m-not-triv} implies that $\mps\neq \{0\}$. 
  
  {\em Step 1}.~ Let $\varphi_1(2^k) := \inf_{j\geq k} \varphi(2^j) \leq \varphi(2^k)$, $k\in\no$, then by Definition~\ref{D-mps}, for any $\lambda\in\mps$,
  \[
\|\lambda | m_{\varphi_1,p}\| \leq \|\lambda | \mps\|, \quad \text{that is,}\quad \mps \hookrightarrow m_{\varphi_1,p}  .
    \]
Plainly, $\{\varphi_1(2^k)\}_{k\in\no}$ is monotonically nondecreasing. Next we show that $\|\lambda |  m_{\varphi_1,p}\| \geq \|\lambda | \mps\|$ for any $\lambda\in m_{\varphi_1,p}$.

Let $j\in\no$, $m\in\Zn$, $k\leq j$, and decompose the cube $Q_{-j,m}$ into $2^{(j-k)\nd}$ cubes $Q_{-k,l}$, then

\begin{align*}
  \lefteqn{\Big(|Q_{-j,m}|^{-1} \sum_{r: Q_{0,r}\subset Q_{-j,m}} |\lambda_r|^p\Big)^{\frac1p}}\\
  & \leq 2^{-j\frac{\nd}{p}} \sup_{l\in\Zn: Q_{-k,l}\subset Q_{-j,m}} \Big(\sum_{r: Q_{0,r}\subset Q_{-k,l}} |\lambda_r|^p\Big)^{\frac1p} 2^{(j-k)\frac{\nd}{p}} \\
&= \sup_{l\in\Zn: Q_{-k,l}\subset Q_{-j,m}} \Big( |Q_{-k,l}|^{-1}\sum_{r: Q_{0,r}\subset Q_{-k,l}} |\lambda_r|^p\Big)^{\frac1p},
\intertext{hence}
                                                                                                                                                                      \lefteqn{\varphi(2^j) \Big(|Q_{-j,m}|^{-1} \sum_{r: Q_{0,r}\subset Q_{-j,m}} |\lambda_r|^p\Big)^{\frac1p}}\\
  & \leq \inf_{k: k\leq j} \varphi(2^j) \sup_{l\in\Zn: Q_{-k,l}\subset Q_{-j,m}} \Big( |Q_{-k,l}|^{-1}\sum_{r: Q_{0,r}\subset Q_{-k,l}} |\lambda_r|^p\Big)^{\frac1p}\\
&\leq \inf_{k: k\leq j} \varphi(2^j) \sup_{l\in\Zn} \Big( |Q_{-k,l}|^{-1}\sum_{r: Q_{0,r}\subset Q_{-k,l}} |\lambda_r|^p\Big)^{\frac1p}\\
&\leq \varphi_1(2^k)  \sup_{l\in\Zn} \Big( |Q_{-k,l}|^{-1}\sum_{r: Q_{0,r}\subset Q_{-k,l}} |\lambda_r|^p\Big)^{\frac1p} \\ 
&\leq \|\lambda | m_{\varphi_1,p}\|,
\end{align*}
for arbitrary $j\in\no$ and $m\in\Zn$, such that finally $\|\lambda | \mps\|\leq \|\lambda | m_{\varphi_1,p}\|$, that is, $m_{\varphi_1,p} \hookrightarrow \mps$.

So we have proved that $m_{\varphi_1,p} = \mps$ (with equality of norms), where $\varphi_1$ is nondecreasing.\\

{\em Step 2}.~ Now we consider the function
\[ \tilde{\varphi}(2^k):= \varphi_1^\ast(2^k) = 2^{k\frac{\nd}{p}} \sup_{j\geq k} \varphi_1(2^j) 2^{-j\frac{\nd}{p}} = \sup_{l\in\no} 2^{-l\frac{\nd}{p}} \varphi_1(2^{k+l}), \quad k\in\no. \]
Let $j\leq k$, then by the monotonicity of $\varphi_1$,
\[
  \tilde{\varphi}(2^{j}) = \sup_{l\in\no} 2^{-l\frac{\nd}{p}} \varphi_1(2^{j+l})
\leq \sup_{l\in\no} 2^{-l\frac{\nd}{p}} \varphi_1(2^{k+l}) = \tilde{\varphi}(2^{k}),
\]
that is, $\tilde{\varphi}$ is nondecreasing, too. We need to verify the latter inequality in \eqref{Gp-def-disc}. Let again $j\leq k$, then
\[
  2^{-j\frac{\nd}{p}} \tilde{\varphi}(2^j) = \sup_{r\geq j} 2^{-r\frac{\nd}{p}} \varphi_1(2^r) \geq \sup_{r\geq k} 2^{-r\frac{\nd}{p}} \varphi_1(2^r) = 2^{-k\frac{\nd}{p}} \tilde{\varphi}(2^k).
\]
Hence $\tilde{\varphi}\in \Gpd $ and it remains to show that $m_{\tilde{\varphi},p} = m_{\varphi_1,p}$ (with equivalence of norms). Obviously, $\tilde{\varphi}(2^k) \geq \varphi_1(2^k)$, $k\in\no$, such that $m_{\tilde{\varphi},p} \hookrightarrow  m_{\varphi_1,p}$. Let $k\in\no$ and $\varepsilon>0$. Then by definition of $\tilde{\varphi}$ again, there exists $k_0=k_0(\varepsilon,k)\in\no$, $k_0\geq k$, such that
\[
  2^{-k\frac{\nd}{p}} \tilde{\varphi}(2^k)\geq \varphi_1(2^{k_0}) 2^{-k_0\frac{\nd}{p}} > (1-\varepsilon) 2^{-k\frac{\nd}{p}} \tilde{\varphi}(2^k).
  \]
  Hence
  \begin{align*}
    \|\lambda | m_{\tilde{\varphi},p} \| & = \sup_{k\in\no, m\in\Zn} \tilde{\varphi}(2^k) 2^{-k\frac{\nd}{p}} \Big(\sum_{r: Q_{0,r}\subset Q_{-k,m}} |\lambda_r|^p\Big)^{\frac1p} \\
                                         & < \frac{1}{1-\varepsilon} \sup_{k_0\in\no, m'\in\Zn} \varphi_1(2^{k_0}) 2^{-k_0\frac{\nd}{p}} \Big(\sum_{r: Q_{0,r}\subset Q_{-k_0,m'}} |\lambda_r|^p\Big)^{\frac1p}\\
    &\leq \frac{1}{1-\varepsilon} \|\lambda|m_{\varphi_1,p}\|, 
  \end{align*}
  which finally results for $\varepsilon\downarrow 0$ in the desired estimate 
  $\|\lambda |  m_{\tilde{\varphi},p}\| \leq \|\lambda |  m_{\varphi_1,p}\|$, i.e., $
 m_{\varphi_1,p} \hookrightarrow     m_{\tilde{\varphi},p}$.
\end{proof}

From now on we shall thus assume that {$\varphi\in\Gpd$}. Moreover, for convenience we shall always assume that $\varphi(1)=1$.
In view of our comparison with the already studied spaces $\ms(\Zn)$, which refers to {$\varphi(2^k)=2^{k\nd/u} \in\Gpd$}, $0<p\leq u<\infty$, this restriction seems natural in this sense, too.  
  
\begin{corollary}\label{lp-mps}
  Let $0<p<\infty$, {$\varphi\in\Gpd$} with $\varphi(1)=1$. Then
  \begin{equation}\label{lp-mfp}
\ell_p \hookrightarrow \mps.
  \end{equation}
\end{corollary}

\begin{proof}
  Let {$\varphi\in \Gpd$}, then \eqref{Gp-def-disc} with $j=0$ implies that $\varphi(2^k)\leq 2^{k\frac{\nd}{p}} =: \varphi_p(2^k)$. Then $m_{\varphi_p,p} = \ell_p$ and Prop.~\ref{Prop-1-mps}(iii) proves the claim.
\end{proof}

{We want to refine the above result \eqref{lp-mfp} in the sense of \cite[Prop.~2.3(iv)]{HaSk_Mo_seq}, see \cite{GuKiSc} for its one-dimensional counterpart. There we had for the function $\varphi(t)\sim t^{\nd/u}$ that $\ell_{u,\infty} \hra \ms$ when $p<u$. Here $\ell_{u,\infty}=\ell_{u,\infty}(\Zn)$, $0<u<\infty$, are the usual Lorentz sequence spaces introduced in \eqref{lorentz-2}.

For some $\varphi\in\Gpd$, let us first introduce the number $\rphi{\varphi}$ parallel to our paper \cite{HLMS}. In view of the monotonicity $\Gpx{p_2}(\mathbb{D})\subseteq\Gpx{p_1}(\mathbb{D})$ for $p_1\leq p_2$, we ask for the `smallest' class $\Gpx{r}(\mathbb{D})$ to which $\varphi\in\Gpd$ belongs, that is, the largest index $r$. Let $\varphi\in \mathcal{G(\mathbb{D})}=\bigcup_{p>0}\Gpd = \lim_{p \to 0} \Gpd$ and 
\begin{align}\label{def-rphi}
		\rphi{\varphi} = \sup\{p:\; \varphi\in\Gpd\},\quad \varphi \in \mathcal{G}(\mathbb{D}).
	\end{align}	
Then 	$0<\rphi{\varphi} <\infty$ or $\rphi{\varphi} =\infty$.  Parallel to \cite[Lemma~4.7]{HLMS} one observes that $\rphi{\varphi}<\infty$ implies  $\varphi\in \mathcal{G}_{\rphi{\varphi}}(\mathbb{D})$, that is, $\rphi{\varphi}=\max\{p:\; \varphi\in\Gpd\}$, while $\rphi{\varphi}=\infty$ if, and only if, $\varphi = \mathrm{const}$. 
We shall assume in the sequel that $\rphi{\varphi} < \infty$.}

{\begin{remark}
	If $\varphi(t) \sim t^{\frac{\nd}{p}} (\log(1+t))^{-a}$, $a>0$, $t\geq 1$, then  $\varphi(2^k) \sim 2^{\frac{k\nd}{p}} k^{-a}$, $k\in \no$. One can easily check that $\rphi{\varphi}=p$. 
\end{remark}}

Now we are ready to give the counterpart of \cite[Prop.~2.3(iv)]{HaSk_Mo_seq}.

\begin{lemma}\label{lorentz-mps}
  Assume $0<p<\infty$, {$\varphi\in\Gpd$} with $\varphi(1)=1$, $ \sup_{k\in\no} \varphi(2^k)=\infty$, and $p<\rphi{\varphi}<\infty$. Then
  \[ \ell_{\rphi{\varphi},\infty} \hra  \mps. \]  
\end{lemma}

\begin{proof}
	Let $\{\lambda^*_\nu\}_{\nu\in\nat}$ be a non-increasing rearrangement of a sequence $\lambda=\{\lambda_k\}_{k\in\Zn}$. Then for any cube $Q_{-j,m}$ we have  
	\begin{align*}
	\Big( \sum_{k:Q_{0,k}\subset Q_{-j,m}}|\lambda_k|^p\Big)^{1/p} &\le \Big( \sum_{\nu=1}^{2^{j\nd}}(\lambda^*_\nu)^p\Big)^{1/p}
	\le \ \sup_{l\in\nat} l^{1/\rphi{\varphi}} \lambda^*_l \, \Big( \sum_{\nu=1}^{2^{j\nd}}\nu^{-p/\rphi{\varphi}}\Big)^{1/p}\\
                                                                       &\le C  \|\lambda|\ell_{\rphi{\varphi}, \infty}\|\, |Q_{-j,m}|^{\frac{1}{p}- \frac{1}{\rphi{\varphi}}} \\
          &= \ C  \|\lambda|\ell_{\rphi{\varphi}, \infty}\|\, |Q_{-j,m}|^{\frac{1}{p}} 2^{-j \frac{\nd}{\rphi{\varphi}}},
	\end{align*}
	since $p<\rphi{\varphi}$ and $|Q_{-j,m}|=2^{j\nd}$. Thus
        \begin{align*} \varphi(2^j) |Q_{-j,m}|^{-\frac{1}{p}} 	\Big( \sum_{k:Q_{0,k}\subset Q_{-j,m}}|\lambda_k|^p\Big)^{1/p}  & \leq C  \|\lambda|\ell_{\rphi{\varphi}, \infty}\|\ \varphi(2^j) 2^{-j \frac{\nd}{\rphi{\varphi}}} \\
          & \leq C \|\lambda|\ell_{\rphi{\varphi}, \infty}\|,\end{align*}
        since $\varphi\in \mathcal{G}_{\rphi{\varphi}}(\mathbb{D})$ when $\rphi{\varphi}<\infty$, and $\varphi(1)=1$. Taking the supremum over $j\in\no$ and $m\in\Zn$ leads to $\|\lambda|\mps\|\leq C\|\lambda|\ell_{\rphi{\varphi},\infty}\|$.
\end{proof}

\begin{remark}
In case of $\varphi(t) \sim t^{\frac{\nd}{u}}$, $\rphi{\varphi}=u$,  Lemma~\ref{lorentz-mps} coincides with \cite[Prop.~2.3(iv)]{HaSk_Mo_seq}.
  \end{remark}

In \cite{HaSk_Mo_seq} we proved that for $\varphi(t) \sim t^{\nd/u}$, $0<p<u<\infty$, the spaces $\mps$
and $c$ or $c_0$, respectively, are incomparable in the sense that for
$\varphi(t) \sim t^{\nd/u}$, $u>p$,
\begin{equation}\label{ddh-1}
  \mps\not\subset c_0, \quad \mps\not\subset c, \quad c_0\not\subset \mps\quad \text{and}\quad c\not\subset \mps,
\end{equation}
while in case of $\varphi(t)\sim t^{\nd/p}$, $0<p=u<\infty$, it is well-known that
\begin{equation}\label{ddh-2}
\mps = \ell_p \subsetneq c_0 \subsetneq c \subsetneq \ell_\infty. 
\end{equation}
So we need to have a closer look on the properties of $\varphi$. Recall that by \eqref{Gp-def-disc} we always have that $\{2^{-k\frac{\nd}{p}} \varphi(2^k)\}_{k\in\no}$ converges to some number $\alpha\in [0,1]$ for $k\to\infty$, parallel to \eqref{lim-exist}. Moreover, in view of Prop.~\ref{Prop-1-mps}(v) we may exclude the case $\sup_{k\in\no} \varphi(2^k)<\infty$. 

\begin{lemma}\label{mps-subset-c}
  Let $0<p < \infty$ and {$\varphi\in \Gpd$} with  $\sup_{k\in\no} \varphi(2^k)=\infty$. 
  \begin{enumerate}[\upshape\bfseries (i)]
	\item
          Assume that $\ds \lim_{k\to\infty} 2^{-k\frac{\nd}{p}} \varphi(2^k) >0$. Then
          \begin{equation}
            \mps = \ell_p \subsetneq c_0 \subsetneq c \subsetneq \ell_\infty.
            \end{equation}
          \item
            Assume that $\ds \lim_{k\to\infty} 2^{-k\frac{\nd}{p}} \varphi(2^k) = 0$. Then  
           the space  $\mps$ is incomparable with $c_0$ and $c$,
            that is, $\mps\not\subset c_0$, $\mps\not\subset c$, $c_0\not\subset \mps$ {and} $c\not\subset \mps$.
          \end{enumerate}
\end{lemma}

\begin{proof}
In case of (i), the assumption leads to $\varphi(2^k) \sim 2^{k\frac{\nd}{p}}$, $k\in\no$, which implies $\mps = \ell_p$ (in the sense of equivalent norms) and \eqref{ddh-2} completes the argument.

We deal with (ii). Consider first the constant sequence $\lambda^1=\{1\}_{k\in\Zn}\in c$. Then 
	\begin{align*}
	\|\lambda^1|\mps\| = & 
                               \sup_{j\in\no; m\in \Zn} \varphi((Q_{-j,m})) |Q_{-j,m}|^{-\frac{1}{p}} |Q_{-j,m}|^{\frac1p} \\
          = & 	\sup_{j\in\no} \varphi((Q_{-j,m})) = \sup_{j\in\no} \varphi(2^{j}) = \infty,
	\end{align*}
	which disproves $c\subset \mps$ (and simultaneously strengthens Proposition~\ref{Prop-1-mps}(iv) by $\mps\subsetneq \ell_\infty$ unless $\sup_{k\in\no} \varphi(2^k)<\infty$).

        Next we want to disprove $c_0\subset\mps$. By assumption, $\sup_{k\in\no} \varphi(2^k)=\infty$, so we can choose a strictly increasing subsequence $\{k_l\}_{l\in\nat}$ such that
        \[
          \lim_{l\to\infty} \varphi(2^{k_l}) = \infty.
          \]
          We may assume that $\varphi(2^{k_l})>0$, such that
          \[
            \lim_{l\to\infty} \varphi(2^{k_l})^{-1} = 0.
            \]
            Now we consider $\tilde{\lambda}= \{\tilde{\lambda}_m\}_{m\in\Zn}$ given by
            \[
\tilde{\lambda}_m= \begin{cases} \varphi(2^{k_l})^{-\frac12}, & |m|\sim 2^{k_l}, \ l\in\nat, \\ 0, & \text{otherwise}.\end{cases}
            \]
Then $\tilde{\lambda}\in c_0$. On the other hand,
	\[
	\|\tilde{\lambda}|\mps\| \geq c\ 
	\sup_{l\in\N} \varphi(2^{k_l}) |Q_{-k_l,0}|^{-\frac{1}{p}} \varphi(2^{k_l})^{-\frac12} |Q_{-k_l,0}|^{\frac1p} = 
	c \sup_{l\in\nat} \varphi(2^{k_l})^\frac12 = \infty,
	\]
	which gives $c_0\not\subset\mps$. 

Finally we present a sequence $\lambda\in\mps \setminus c \subsetneq \mps\setminus c_0$. Recall that $\ds \lim_{k\to\infty} 2^{-k\frac{\nd}{p}} \varphi(2^k) = 0$. Then 
we can construct  a strictly increasing  sequence $\{n_\ell\}_{\ell\in\nat}$, $n_{\ell}\in \no$, such that 
\begin{equation}\label{2.06-2}
	\varphi(2^{n_{\ell}})2^{-\frac{\nd n_{\ell}}{p}} \le {\ell^{-\frac{1}{p}}}, 
	\quad \ell\in  \N, 
\end{equation}
where we can take  $n_1=0$. We consider the special sequence $\lambda = \{\lambda_k\}_{k\in\Zn}$ given by,
	\begin{equation}\label{2.06-2a} 
	\lambda_k = \begin{cases} 1, & \text{if}\ k=(2^{n_\ell}, 0, \dots, 0)\ \text{for some}\  l\in\nat, \\ 0, & \text{otherwise}.
	\end{cases}
	\end{equation}
	Obviously $\lambda\not\in c$, in particular, $\lambda\not\in c_0$. Moreover, by construction,
	\begin{align*}
          \|\lambda|\mps\| & \leq c\ \sup_{l\in\nat} \varphi(2^{n_\ell}) |Q_{-n_\ell,0}|^{-\frac1p} \Big(\sum_{k:Q_{0,k}\subset Q_{-n_\ell,0}} \!\!\!|\lambda_k|^p\Big)^{\frac1p}\\
          &\leq c'\ \sup_{l\in\nat} \varphi(2^{n_\ell}) 2^{-{\frac{\nd n_{\ell}}{p}}} l^{\frac1p} < \infty.
	\end{align*}
	So the subspaces $c_0$, $c$ and $\ms$ of $\ell_\infty$ are incomparable in the above sense.
  \end{proof}

\begin{remark}
  Note that the above Lemma~\ref{mps-subset-c} obviously generalises our findings \eqref{ddh-1}, \eqref{ddh-2} in the special cases $\varphi(t) = t^{\nd/u}$ for $u\geq p$.   
\end{remark}

The next question we approach is the separability of $\mps$. This obviously depends on $\varphi$, as we already know that in case of $\varphi(t)=t^{\nd/p}$, $0<p<\infty$, the space $\mps=\ell_p$ is separable, while for $\varphi(t)=t^{\nd/u}$, $0<p<u<\infty$, the space $\ms$ is not separable, cf. \cite[Prop.~2.9]{HaSk_Mo_seq}. The same happens whenever $\sup_{k\in\no} \varphi(2^k)<\infty$, as then by Proposition~\ref{Prop-1-mps} we arrive at $\mps=\ell_\infty$ which is non-separable, too. So we may assume in the sequel that $ \sup_{k\in\no} \varphi(2^k)=\infty$.

\begin{proposition}
  Let $0<p<\infty$, {$\varphi\in\Gpd$} with $\varphi(1)=1$, and assume $\  \sup_{k\in\no} \varphi(2^k)=\infty$. Then $\mps$ is separable if, and only if, $\ds \lim_{k\to\infty} 2^{-k\frac{\nd}{p}} \varphi(2^k) >0$.
        \end{proposition}

\begin{proof}
 Assume first that $\ds \lim_{k\to\infty} 2^{-k\frac{\nd}{p}} \varphi(2^k) >0$. Then it follows in the same way as in Lemma~\ref{mps-subset-c}(i) that $\mps=\ell_p$ which is well-known to be separable.
  
Now assume that $\ds \lim_{k\to\infty} 2^{-k\frac{\nd}{p}} \varphi(2^k) = 0$, and recall the subsequence given by \eqref{2.06-2}. Let $E$ be a subset of $\N$. We consider the following sequences $\lambda^{(E)}$ defined by 
	\begin{equation}\nonumber
	\lambda^{(E)}_k = 
	 \begin{cases}
           1 & \text{if}\qquad k=(2^{n_\ell},0,\ldots , 0) \ \text{and}\ {\ell \in E},  \\
           0 & \text{otherwise}. 	 
	 \end{cases}
	\end{equation}
	It should be clear that 
	\[ \|\lambda^{(E)}|\mps\| \le  \|\lambda^{(\N)}|\mps\|.\]
	So all the sequences belong  to $\mps(\Zn)$ if $\lambda^{(\N)}\in \mps(\Zn)$. For any cube $Q_{-j,m}$ we have by construction,
	\begin{align*}
    \varphi(2^j) |Q_{-j,m}|^{-\frac{1}{p}} \Big(\sum_{k:Q_{0,k}\subset Q_{-j,m}} |\lambda^{(\N)}_k|^p \Big)^\frac{1}{p} 
    &\le  \varphi(2^j)	2^{-j\frac{\nd}{p}} \Big(\sum_{k:Q_{0,k}\subset Q_{-j,0}} |\lambda^{(\N)}_k|^p \Big)^\frac{1}{p} \\
    &\le \sup_{l\in\nat} \varphi(2^{n_\ell}) 2^{-n_\ell\frac{\nd}{p}} \ell^\frac{1}{p} \le C <\infty. 
	\end{align*}
	Hence $\lambda^{(\N)}\in \mps(\Zn)$. If $E$ and $F$ are different subsets of $\N$, then 
	\[ \|\lambda^{(E)}- \lambda^{(F)}|\mps\| \ge 1 . \] 
	Thus $\mps(\Zn)$ contains a non-countable set of sequences such that the distance between two different elements of it is at least $1$.    
\end{proof}

Next we define a closed proper subspace of $\mps$ as follows. Let $\mps^{00}=\mps^{00}(\Zn)$ be the closure of $c_{00}$ in $\mps$,
\[
\mps^{00} = \overline{c_{00}}^{\|\cdot|\mps\|}.
\]
Obviously $\mps^{00}$ is separable. Moreover, let us denote by $\mps^0= \mps^0(\Zn)$ the subspace of null sequences which belong to $\mps$,
\[
\mps^0= \ms  \cap c_0\ . 
\]
Then we have the following basic properties.

\begin{corollary}
	Let $0<p<\infty$, {$\varphi\in\Gpd$} with $\varphi(1)=1$, and $ \sup_{k\in\no} \varphi(2^k)=\infty$.
	The following conditions are equivalent:
	\begin{enumerate}[\upshape\bfseries (i)]
		\item
		$\ds \lim_{k\to\infty} 2^{-k\frac{\nd}{p}} \varphi(2^k) >0$,
		\item $\mps \hookrightarrow c_0$,
		\item $\mps = \ell_p$,
		\item $\mps^{00} =\mps$
	\end{enumerate}
\end{corollary}

\begin{proof}
While the equivalence of the first three assertions follows from Lemma \ref{mps-subset-c}, the fourth statement is an immediate consequence of  the third one. So it remains to show that it implies the first assertion. If (i) does not hold, then   	$\ds \lim_{k\to\infty} 2^{-k\frac{\nd}{p}} \varphi(2^k) = 0$ and the sequence \eqref{2.06-2a} belongs $\mps$, but not to $\mps^{00}$. This contradicts (iv).
\end{proof}

We can also give the counterpart of \cite[Lemma~4.6]{HaSk_Mo_seq}.

\begin{lemma}
	Assume $0<p<\infty$, {$\varphi\in\Gpd$} with $\varphi(1)=1$, $ \sup_{k\in\no} \varphi(2^k)=\infty$, and $\ds \lim_{k\to\infty} 2^{-k\frac{\nd}{p}} \varphi(2^k) =0$. Then $\mps^0$ and $\mps^{00}$ are proper closed subspaces of $\mps$, with 
	\[ \mps^{00}\subsetneq\mps^0 \subsetneq \mps . \]
\end{lemma}

\begin{proof}
	By definition, $\mps^{00}$ is a closed subspace of $\mps$. The fact that $\mps^0 \subsetneq \mps$ is a proper subspace of $\mps$ follows from Proposition~\ref{Prop-1-mps}. We show that $\mps^0$ is closed in $\mps$,
	\[
	\mps^0 = \overline{\mps^0}^{\|\cdot|\mps\|}.
	\]
	Clearly $\mps^0 \subseteq \overline{\mps^0}^{\|\cdot|\mps\|}$, so we have to verify the converse inclusion. Let $\lambda=\{\lambda_k\}_{k\in\Zn}\in\overline{\mps^0}^{\|\cdot|\mps\|}$ and $\ve>0$ be arbitrary. Then, by definition, there exists some $\mu=\{\mu_k\}_{k\in\Zn}\in\mps^0$ such that
	\begin{equation}\label{mps-12}
	\|\mu-\lambda|\mps\| = \sup_{j\in\no, m\in\Zn} \varphi(2^j) |Q_{-j,m}|^{-\frac1p} \Big(\sum_{k:Q_{0,k}\subset Q_{-j,m}} \!\!\!|\mu_k-\lambda_k|^p\Big)^{\frac1p} < \ve.
	\end{equation}
	Now $\mu\in\mps^0\subset \mps$ and $\mps$ is complete, thus $\lambda\in\mps$. Moreover, applying \eqref{mps-12} with $j=0$ implies that 
	\[
	\|\mu-\lambda|\ell_\infty\| = \sup_{k\in\Zn} |\mu_k-\lambda_k|<\ve.
	\]
	However, $\mu\in\mps^0 \subset c_0$ thus leads to $\lambda\in c_0$. So finally $\lambda\in \mps \cap c_0 = \mps^0$.

	It remains to verify that $\mps^{00}\subsetneq\mps^0$. First note that, by definition, $\mps^{00} \subseteq \mps^0$.  
	 {To prove that $\mps^{00}\not=\mps^0$ we consider} special lattice points $m_j=(2^{n_j},0,\ldots , 0)\in\Zn$, $j\in\no$ and $n_j\in\no$, that will be specified below. If the cubes $ Q_{-j,m_j}$ are pairwise disjoint, then we can  put
	\begin{equation}\label{13_09_1}
	\lambda_{k} = \begin{cases}
		\varphi(2^j)^{-1} & \text{if}\quad Q_{0,k}\subset Q_{-j,m_j}, \\
		0 & \text{otherwise}.
	\end{cases}
	\end{equation}
	We define the points $m_j$ by induction. We put $n_0=1$. Assume that $n_{j}$ has already been defined. By assumptions  $\ds \lim_{k\to\infty} 2^{k\frac{\nd}{p}} \varphi(2^k)^{-1} =\infty$, so we can always find $\nu_j\in \no$ such that 
	\begin{equation}\label{13_09_2}
		\left(\sum_{\ell=0}^j \varphi(2^\ell)^{-p}{2^{\ell d}} \right)^\frac{1}{p} \le 2^{\nu_j\frac{\nd}{p}} \varphi(2^{\nu_j})^{-1}.
	\end{equation}  
	{We may assume that the sequence $\nu_j$ is strictly increasing and that $\nu_j>j$ for any $j\in \no$. Let us put $n_{j+1}= n_j+\nu_j$.} The cubes  $ Q_{-j,m_j}$ are pairwise disjoint {since $2^{n_{j+1}}-2^{n_{j}}= 2^{n_j}(2^{\nu_j}-1)\ge 2^{n_j}2^{j}$.  So} the sequence \eqref{13_09_1} is well-defined. Moreover $n_j\rightarrow\infty$ for $j\to\infty$, so  $\lambda\in c_0$. 	It remains to prove that $\lambda\in \mps$ {since $\lambda\notin \mps^{00}$ should be  obvious.} 
	
        {In order to estimate the norm of $\lambda$ in $\mps$, we should consider all dyadic  cubes $Q_{-\mu,m}$, $\mu\in \no$. However, only  those cubes that have a non-empty intersection with some of the cubes  $Q_{-j,m_j}$ give some contribution.  All the cubes  used in the construction are dyadic, therefore we have to consider two cases: $Q_{-\mu,m}$ is contained in one of the  cubes $Q_{-j,m_j}$, or $Q_{-\mu,m}$ contains a finite subfamily of these cubes.}
        
If $Q_{-\mu,m}\subset Q_{-j,m_j}$ for some $j$, then 
\begin{equation}\label{13_09_3}
	 \varphi(2^\mu) |Q_{-\mu,m}|^{-\frac1p} \Big(\sum_{k:Q_{0,k}\subset Q_{-\mu,m}} \!\!\!|\lambda_k|^p\Big)^{\frac1p} \le \varphi(2^\mu) \varphi(2^j)^{-1}\le 1,   
\end{equation}
{since $\mu\le j$ and $\varphi$ is non-decreasing. In the other case $Q_{-\mu,m}$ contains finitely many  dyadic cubes $Q_{-j,m_j}$.}  Let $J$ denote the largest $j$ such that $Q_{-j,m_j}\subset Q_{-\mu,m}$. If $\mu\le \nu_{J-1}$, then the cube $Q_{-\mu,m}$ {contains only one of the cubes} $Q_{-j,m_j}$, i.e. $Q_{-J,m_J}$, since $2^{\nu_{J-1}}\le 2^J+(2^{n_J}-2^{n_{J-1}})$. In that case 
\begin{equation}\label{13_09_4}
{ \varphi(2^\mu) |Q_{-\mu,m}|^{-\frac1p} \Big(\sum_{k:Q_{0,k}\subset Q_{-\mu,m}} \!\!\!|\lambda_k|^p\Big)^{\frac1p} \le 
 \frac{\varphi(2^\mu)|Q_{-J,m_J}|^{\frac1p}}{\varphi(2^J)|Q_{-\mu,m}|^{-\frac1p}} \le 1 }
\end{equation}
since $\mu>J$, recall \eqref{Gp-def-disc}. If $\mu> \nu_{J-1}$, then  
\begin{align}
  \lefteqn{ \varphi(2^\mu) |Q_{-\mu,m}|^{-\frac1p} \Big(\sum_{k:Q_{0,k}\subset Q_{-\mu,m}} \!\!\!|\lambda_k|^p\Big)^{\frac1p}}\nonumber\\
  & \le 
    \varphi(2^\mu) |Q_{-\mu,m}|^{-\frac1p} \left(\sum_{\ell=0}^J \varphi(2^\ell)^{-p}2^{jd} \right)^\frac{1}{p} \nonumber  \\
&\leq  \varphi(2^\mu) |Q_{-\mu,m}|^{-\frac1p} \left(2^{\nu_{J-1}\frac{\nd}{p}} \varphi(2^{\nu_{J-1}})^{-1}+\frac{|Q_{-J,m_J}|^{\frac1p}}{\varphi(2^J)}\right)\le 2, \label{13_09_5}
\end{align}
where the second inequality follows from \eqref{13_09_2} and we finally apply \eqref{Gp-def-disc} again. Now we get  from \eqref{13_09_2}-\eqref{13_09_5} that $\lambda\in\mps$,  but obviously $\lambda\notin \mps^{00}$. 
\end{proof}

\begin{remark}
  	The above result sheds some further light on the difference of the two norms $\|\cdot|\ell_\infty\| $ and $\|\cdot|\mps\|$, since in the classical setting 
	\[ \overline{c_{00}}^{\|\cdot|\ell_\infty\|} = c_0
	\]
	is well-known, in contrast to $\mps^{00} \subsetneq \mps^0$. This generalises our results from \cite{HaSk_Mo_seq} in case of $\varphi(t) \sim t^{\nd/u}$, $0<p\leq u<\infty$.
\end{remark}

        \section{Embeddings}\label{sec-2}
        We prove our main result about embeddings of different generalised Morrey sequence spaces. Here we also use and adapt some ideas of our paper \cite{HaSk-bm1}, cf. the proof of Theorem 3.2 there. A similar construction was used by P.~Olsen in \cite{Olsen95}, cf. the proof of Theorem 10 in \cite{Olsen95}.

We study the embedding 
\[
  \mpse\hookrightarrow \mpsz
\]
for arbitrary $0<p_i<\infty$, $\varphi_i\in\Gpx{p_i}(\mathbb{D})$, $i=1,2$.

Note first that in case of $\sup_{k\in\no} \varphi_2(2^k)<\infty$ we get $\mpsz=\ell_\infty$, cf. Prop.~\ref{Prop-1-mps}(v), and thus $\mpse\hookrightarrow \ell_\infty$ without further assumption on $p_1$, $\varphi_1$, recall Prop.~\ref{Prop-1-mps}(iv). So the condition $\sup_{k\in\no} \varphi_2(2^k)<\infty$ is sufficient for the considered embedding.

On the other hand, if we assume  $\sup_{k\in\no} \varphi_1(2^k)<\infty$ and the embedding to hold, we arrive at $\ell_\infty = \mpse \hookrightarrow \mpsz \hookrightarrow \ell_\infty$, that is, $\mpse=\ell_\infty = \mpsz$, and thus again $\sup_{k\in\no} \varphi_2(2^k)<\infty$ by Corollary~\ref{linf-mps}.

So we may and shall assume in the sequel that $\sup_{k\in\no} \varphi_i(2^k)=\infty$, $i=1,2$.

\begin{theorem}  \label{cont}
	Let    $0<p_i<\infty$ and  $\varphi_i\in\Gpx{p_i}(\mathbb{D})$, with $\varphi_i(1)=1$, $i=1,2$. Let $\varrho=\min\{1,\frac{p_1}{p_2}\}$.  Assume $\sup_{k\in\no} \varphi_i(2^k)=\infty$, $i=1,2$. Then the embedding 
	\begin{equation} \label{mps1cont}
	\mpse\hookrightarrow \mpsz
	\end{equation}
	is continuous  if, and only if, 
	the following condition holds,
\begin{equation}\label{cont-mps}
	\sup_{j\in \No} \frac{\varphi_2(2^{j})}{\varphi_1(2^j)^\varrho}<\infty.
\end{equation}
	The embedding \eqref{mps1cont} is never compact.
\end{theorem}


\begin{proof}
	\noindent{\em Step 1}.~ First we prove the sufficiency of the conditions. 
	Let us assume that $\varrho=1$, i.e., $p_1\ge p_2$.   If $\varphi_2= c\,\varphi_1$ for some positive constant $c>0$, then  by \eqref{cont-mps}  the result follows by the same argument as in Prop.~\ref{Prop-1-mps}(ii). In the general case we have $\varphi_2 \le  c\,\varphi_1$ for some  $c>0$.  Let  $\widetilde{\varphi}_2= c\,\varphi_1$. Then 
	\[ m_{\varphi_1,p_1} \hookrightarrow m_{\widetilde{\varphi}_2,p_2} \hookrightarrow \mpsz , \]
	where the last embedding follows from Prop.~\ref{Prop-1-mps}(iii).  \\	 
	Let now $p_1 < p_2 $. 
	Assume first  $\varphi_2(2^j)^{p_2}=c\,\varphi_1(2^j)^{p_1}$, $c>0$. 
	Then for any  $j\in\no$ and $m\in\Zn$ we get 
	\begin{align*}
	\varphi_2(2^j) 2^{-j\frac{\nd}{p_2}} & \Big(\sum_{k:Q_{0,k}\subset Q_{-j,m}}\!\!\!|\lambda_{k}|^{p_2}\Big)^\frac{1}{p_2} \\
	 \le  c\,  &\left[\varphi_1(2^j) 2^{-j\frac{\nd}{p_1}}\Big(\sum_{k:Q_{0,k}\subset Q_{-j,m}}\!\!\!|\lambda_{k}|^{p_1}\Big)^\frac{1}{p_1} \right]^\frac{p_1}{p_2} \sup_{k:Q_{0,k}\subset Q_{-j,m}}|\lambda_{k}|^{1-\frac{p_1}{p_2}}.	
	\end{align*}  
	The last inequality implies that
	\begin{align}
	\|\lambda|\mpsz\| \le c\,  \|\lambda|\mpse\|^{\frac{p_1}{p_2}}\, \|\lambda |\ell_\infty\|^{1-\frac{p_1}{p_2}} \le  \,c\, \|\lambda|\mpse\|,
	\label{mps-2}
	\end{align}
by Prop.~\ref{Prop-1-mps}(iv).
If $\varphi_2(2^j)^{p_2}\le c\varphi_1(2^j)^{p_1}$, then choose $\tilde{\varphi}_1 := c^{-\frac{1}{p_1}}\varphi_2^{\frac{p_2}{p_1}}<\varphi_1$. Note that $\tilde{\varphi}_1\in\Gpx{p_1}(\mathbb{D})$.  Then by the preceding argument, $m_{\tilde{\varphi}_1,p_1} \hookrightarrow \mpsz$, while Prop.~\ref{Prop-1-mps}(iii)
leads to $\mpse \hookrightarrow m_{\tilde{\varphi}_1,p_1}$.\\


	\noindent{\em Step 2}.~   We consider the  necessity of the conditions. Assume that 
	the  condition  \eqref{cont-mps} is not satisfied. Then there exists an increasing  subsequence of positive integers $\{j_\ell\}_{\ell\in\nat}$ such that 
	\begin{equation}\label{neces1}
		\ell\, \varphi_1(2^{j_\ell})^\varrho <  \varphi_2(2^{j_\ell}), \qquad \ell\in\nat. 
	\end{equation}
	
	{\em Substep 2.1.}~ First we assume that  $p_1\ge p_2$, that is, $\varrho=1$, and that the  condition  \eqref{cont-mps} is not satisfied. Then there exists an increasing  subsequence of positive integers $\{j_\ell\}_{\ell\in\nat}$ such that $\ell\, \varphi_1(2^{j_\ell}) <  \varphi_2(2^{j_\ell})$, $\ell\in\nat$.
     Let $\lambda=(\lambda_k)_k$ be the sequence given by the formulae \eqref{13_09_1} with $\varphi=\varphi_1$ and $p=p_1$. Then $\|\lambda|\mpse\|<\infty$, but 
 on the other hand,  
 \begin{align*}
   \sup_{\ell\in\nat} \varphi_2(2^{j_\ell}) |Q_{-j_\ell,m_\ell}|^{-\frac{1}{p_2}} & \Big(\sum_{k:Q_{0,k}\subset Q_{-j_\ell,m_\ell}}\!\!\! |\lambda_k|^{p_2}\Big)^{1/p_2}\\
   & \ =\ \sup_{\ell\in\nat}\, \varphi_2(2^{j_\ell}) \varphi_1(2^{j_\ell})^{-1} = \infty ,\end{align*}
in view of \eqref{neces1}. So $\{\lambda_k\}_k$ does not belong to $\mpsz$. \\
	
	{\em Substep 2.2.}~ Now we assume that  $p_1 < p_2$.  For any $j\in \N$ we put 
	$$k_j = \whole{2^{dj}\varphi_1(2^j)^{-p_1}},$$
	recall $\whole{x}=\max\{l\in\Z: l \leq x\}$. Then $1\le k_j < 2^{dj}$ since $\varphi_1\in \Gpx{p_1}(\mathbb{D}) $. Moreover, one can easily check that   there is a constant $c_{p_1}$ such that 
	\begin{equation}\label{30.12-1}
	k_j \le \ c_{p_1}\ 2^{d(j - \nu)} k_\nu, \qquad \text{if}\qquad 1\le \nu < j \, .
	\end{equation}
	For convenience let us assume that $c_{p_1}= 1$ (otherwise the argument below has to be modified in an obvious way). 
	For any  $j\in \N$ we define a sequence $\lambda^{(j)} = \big\{\lambda^{(j)}_{m}\big\}_{m\in\Zn}$ in the following way. We assume that $k_j$ elements of the sequence  equal  $1$ and the rest is equal to  $0$.  If $Q_{0,m}\nsubseteq Q_{-j,0}$, then we put $\lambda^{(j)}_{m}= 0$. Moreover,  because of the inequality \eqref{30.12-1}, we can choose the elements that equal $1$ is such a way that the following property holds
	\begin{align*}  
	&\text{if}\quad Q_{-\nu,n}\subseteq Q_{-j,0} \quad \text{and} \quad Q_{-\nu,n} = \bigcup_{i=1}^{2^{d\nu}} Q_{0,m_i}, \\ 
	&\text{then at most}\; k_\nu\; \text{elements}\; \lambda^{(j)}_{{m_i}}\; \text{equal}\; 1.  
	\end{align*}
	By construction, if $Q_{-\nu,n}\subseteq Q_{-j,0}$, then  
	\begin{align}	
	\sum_{k:Q_{0,m}\subset
		Q_{-\nu,n}}\!\!\!|\lambda^{(j)}_{m}|^{p_1} \le k_\nu \le 2^{d\nu} \varphi_1(2^\nu)^{-p_1}
	\label{30.12-2}
	\end{align}
	and  the last sum is equal to $k_\nu$ if $\nu=j$. Thus 
	\begin{equation}\label{30.12-3}
	\|\lambda^{(j)}|\mpse\|\,\le \, 1 .
	\end{equation}
	Furthermore,  $\{2^{dj}\varphi_1(2^j)^{-p_1}\}_j$ is an increasing sequence since $\varphi_1\in \Gpx{p_1}(\mathbb{D})$. If this sequence is divergent to $\infty$, then $k_j\sim  2^{dj}\varphi_1(2^j)^{-p_1} \to\infty$ for $j\to\infty$. On the other hand, if the sequence converges to some $c\ge 1$, then $k_j\sim c$ and $2^{-dj}\sim \varphi_1(2^j)^{-p_1}$ for sufficiently large $j$. So, taking the sequence $\{j_\ell\}_\ell$ from \eqref{neces1} we get 
	\begin{align}		\nonumber
		\|\lambda^{(j_\ell)}|\mpsz\| &\ge\  \varphi_2(2^{j_\ell}) \left(\frac{1}{|Q_{-j_\ell,0}|}	\sum_{k:Q_{0,m}\subset
			Q_{-\nu,n}}\!\!\!|\lambda^{(j_\ell)}_{m}|^{p_2} 
                                               \right)^{1/p_2} \\
          & =  \varphi_2(2^{j_\ell}) \big( 2^{-dj_\ell} k_{j_\ell}\big)^{1/p_2}  \\  & \ge \ C  \frac{\varphi_2(2^{j_\ell})}{\varphi_1(2^{j_\ell})^{p_1/p_2}}\ge C\ell
\label{necs-2}
	\end{align}
	for sufficiently large $j_\ell$, cf. \eqref{neces1}.

	However, since we assume that the embedding \eqref{mps1cont} holds, there is a positive constant $c>0$ such that    
	$$
	\|\lambda^{(j)}|\mpsz\| \,\le \, c\, 	\|\lambda^{(j)}|\mpse\|\,\le \, c \qquad \text{for any}\qquad \lambda^{(j)}\in \mpse, \quad j\in\nat\, .
	$$
	In view of the last inequalities  \eqref{30.12-3} and \eqref{necs-2} we get 
	$$
C	\ell\, \le \, \|\lambda^{(j_\ell)}|\mpsz\|\,   \le c  \|\lambda^{(j_\ell)}|\mpse\|\, \leq \, c, 
	$$
	and this leads to a contradiction for large  $\ell$.\\

	\noindent{\em Step 3}.~ The non-compactness of \eqref{mps1cont} immediately follows from Proposition~\ref{Prop-1-mps}(iv) and Corollary~\ref{lp-mps},
	\[
	\ell_{p_1} \hookrightarrow \mpse \hookrightarrow \mpsz\hookrightarrow \ell_\infty
	\]
	and the non-compactness of $\ell_{p_1}\hookrightarrow \ell_\infty$.
\end{proof}

\begin{remark}
It is obvious that our considerations in front of Theorem~\ref{cont} fit well together with the result: if $\sup_{k\in\no} \varphi_2(2^k)<\infty$, then \eqref{cont-mps} is always satisfied since $\varphi_1\in\Gpx{p_1}(\mathbb{D})$ is increasing, so \eqref{cont-mps} is sufficient in that case. In case of $\sup_{k\in\no} \varphi_1(2^k)<\infty$ condition \eqref{cont-mps} can only be satisfied when $\sup_{k\in\no} \varphi_2(2^k)<\infty$, too, as discussed above.
\end{remark}

\begin{example}
  Let $0<p_i\leq u_i<\infty$ and assume that $\varphi_i(2^k) \sim 2^{k\nd/u_i}$, $k\in\no$, $i=1,2$. Then $\varphi_i\in\Gpx{p_i}(\mathbb{D})$, $i=1,2$, $\mpse = \mse$ and $\mpsz=\msz$. Theorem~\ref{cont} reads in this case as
  \[
    \mse\hookrightarrow \msz \qquad\text{if, and only if,}\qquad
\frac{u_1}{u_2} \leq \min\left\{1, \frac{p_1}{p_2}\right\},
  \]
  with the embedding being never compact. This coincides with our recent finding in \cite{HaSk-bm1}.
\end{example}

We collect some immediate consequences of Theorem~\ref{cont}. 

\begin{corollary}\label{Co-same}
Let    $0<p_i<\infty$, $\varphi_i\in\Gpx{p_i}(\mathbb{D})$, with $\varphi_i(1)=1$, and assume $\ \sup_{k\in\no} \varphi_i(2^k)=\infty$, $i=1,2$. Then $\mpse = \mpsz$ (in the sense of equivalent norms) if, and only if, $\varphi_1 \sim \varphi_2$ and $p_1=p_2$.
\end{corollary}

\begin{remark}
Note that in case of $\sup_{k\in\no} \varphi_i(2^k)<\infty$ for $i=1$ or $i=2$, no further condition on $p_i$ can be extracted in view of  Corollary~\ref{linf-mps}.
\end{remark}

Dealing with the situation $\varphi_1=\varphi_2=\varphi$, we can now strengthen Proposition~\ref{Prop-1-mps}(ii).

  \begin{corollary}
  Let    $0<p_i<\infty$, $i=1,2$, $\varphi\in\Gpx{\max(p_1,p_2)}(\mathbb{D})$, with $\varphi(1)=1$.  Consider the embedding
  \begin{equation}
    \label{same-phi}
\id:    m_{\varphi, p_1}\to m_{\varphi, p_2}.
  \end{equation}
    \begin{enumerate}[\upshape\bfseries (i)]
 \item
    Assume that $\sup_{k\in\no} \varphi(2^k)<\infty$. Then \eqref{same-phi} is always continuous, in particular, $m_{\varphi,p_1}=m_{\varphi, p_2}$.
  \item
    Assume that $\sup_{k\in\no} \varphi(2^k)=\infty$. Then the embedding \eqref{same-phi} is continuous  if, and only if,  $p_1\geq p_2$.
\item
    The embedding \eqref{same-phi} is never compact.
\end{enumerate}
\end{corollary}

  \begin{proof}
Part (ii) directly follows from Theorem~\ref{cont}, while part (i) is a consequence of $m_{\varphi,p_i} = \ell_\infty$, $i=1,2$, recall Corollary~\ref{linf-mps}.
\end{proof}

Finally, let us compare the spaces $\mps$ with the classical scale $\ell_r$. Clearly, if $\sup_{k\in\no} \varphi(2^k)<\infty$, then $\mps\hra\ell_r$ if, and only if, $r=\infty$,  while $\ell_r \hra\mps$ is true for all $0<r\leq\infty$. Here we apply Corollary~\ref{linf-mps} again. Now we deal with the remaining case, that is, when $\sup_{k\in\no} \varphi(2^k)=\infty$.

\begin{corollary}\label{corlp}
  Let    $0<p<\infty$, $\varphi\in\Gpd$, with $\varphi(1)=1$, and assume  $\ \sup_{k\in\no} \varphi(2^k)=\infty$. Let $0<r\leq\infty$.
 \begin{enumerate}[\upshape\bfseries (i)]
 \item
   Then $\mps\hra \ell_r$ if, and only if, $p\leq r$ and $\varphi(2^j)\sim 2^{j\nd/p}$, that is, $\mps=\ell_p$.
 \item
   Then $\ell_r\hra \mps$ if, and only if,
   \[r\leq p, \qquad \text{or}\quad r>p\quad\text{and}\quad \{2^{-j\nd/r} \varphi(2^j)\}_{j\in\no} \in\ell_\infty.\]
\end{enumerate}  
\end{corollary}

\begin{proof}
  The result immediately follows from Theorem~\ref{cont} using $\ell_r=m_{\varphi_r,r}$, $\varphi_r(2^j)\sim 2^{j\nd/r}$, $j\in\no$, together with the assumption $\varphi\in \Gpd$. 
\end{proof}

\begin{remark}
  This is the extension of \cite[Cor.~3.3(ii)]{HaSk-bm1} when we dealt with the case  $\varphi(2^j)\sim 2^{j\nd/u}$, $j\in\no$, $0<p\leq u<\infty$. Then the condition in (ii) simplifies to $u\geq r$.   Note that $r=\infty$ is excluded in (ii) in coincidence with Corollary~\ref{linf-mps}. Part (i) can also be considered as the discrete counterpart of \eqref{Mphi-Lr}.
  \end{remark}

 \begin{corollary}
		Let $\varphi\in \mathcal{G}(\mathbb{D})$ and $\rphi{\varphi}<\infty$.  Then                 for any $p \in (0, \rphi{\varphi}]$,
\[ m_{\rphi{\varphi}, p}\hra m_{\varphi, p} . \]
              \end{corollary}

\begin{proof}
The statement follows directly from the inclusion $\varphi\in \mathcal{G}_{\rphi{\varphi}}(\mathbb{D})$ in case of $\rphi{\varphi}<\infty$, since by the definition of the $\mathcal{G}_{\rphi{\varphi}}(\mathbb{D})$ class we get $2^{k\nd/\rphi{\varphi}} \ge \varphi(2^k)$, $k\in \no$.
\end{proof}

\begin{remark}
    The above result can be seen as the extension of both, Corollary~\ref{lp-mps} (since $\ell_p\hra  m_{\rphi{\varphi}, p}$ in case of $p\leq \rphi{\varphi}$), and Lemma~\ref{lorentz-mps} in case of $p<\rphi{\varphi}$ (since $\ell_{u,\infty}\hra m_{u,p}$ whenever $u>p$, cf. \cite[Prop.~2.3(iv)]{HaSk_Mo_seq}). \\
    In case of $\rphi{\varphi}=\infty$, not covered by the above corollary, we are led to the situation that $\varphi=\mathrm{const}$, hence by Proposition~\ref{Prop-1-mps}(v) and Corollary~\ref{linf-mps}, $m_{\rphi{\varphi}, p}= \ell_\infty= m_{\varphi, p} $, extending the above result. 
\end{remark}

\section{Finite-dimensional generalised  Morrey sequence spaces}\label{sec-3}

We insert a short digression to finite-dimensional spaces where, in contrast to the spaces $\mps$ considered so far, all continuous embeddings are compact. But now we are interested in the dependence of the norms of such compact embedding operators on their dimension. For simplicity we restrict ourselves from the very beginning to dyadic dimensions. We follow our approach in \cite{HaSk-bm1}.

\begin{definition}
	Let $0<p <\infty$, $j\in\no$ and  $\varphi\in \Gpd$ be fixed and $\mathcal{K}_j = \{k\in \Zn: Q_{0,k}\subset Q_{-j,0}\}$ . We define
	\begin{align}
		\mpb = & \Big\{ \lambda = \{\lambda_{k}\}_{k\in \mathcal{K}_j} \subset\C: \nonumber\\
		&\quad  \|\lambda|\mpb\| = \sup_{Q_{-\nu,m}\subset Q_{-j,0}}\!\!
		 \varphi(2^{\nu})\Big(|Q_{-\nu,m}|^{-1}\sum_{k:Q_{0,k}\subset Q_{-\nu,m}}\!\!|\lambda_k|^p \Big)^{\frac 1 p}<\infty\Big\},
	\end{align}
	where the  supremum is taken over all $\nu\in\no$ and $m\in\Zn$ such that 
	$Q_{-\nu,m}\subset Q_{-j,0}$.\nonumber
\end{definition}

\begin{proposition}\label{lemma15030}
	Let $0< p_1, p_2<\infty$,   $\varphi_i\in \Gpx{p_i}(\mathbb{D})$, $i=1,2$  and  $j\in \no$  be given. Then the norm of the compact identity operator  
	\begin{equation}\label{id_j-m}
		\id_j: \mpbet\hookrightarrow \mpbzt
	\end{equation}
	satisfies
	\begin{equation}\label{1503-0}
		\|\id_j\| 
		=  	\max_{0\le\nu\le j} \frac{\varphi_2(2^{\nu})}{\varphi_1(2^\nu)}
	\end{equation}
     provided that  $p_1\ge p_2$. If $p_1 < p_2$, then  there is a constant $c$, $0<c\le 1$, independent of j such that 
	\begin{equation}\label{1503-a}
		c\, 	\max_{0\le\nu\le j} \frac{\varphi_2(2^{\nu})}{\varphi_1(2^\nu)^{p_1/p_2}} \le\|\id_j\| \leq 
		\max_{0\le\nu\le j} \frac{\varphi_2(2^{\nu})}{\varphi_1(2^\nu)^{p_1/p_2}}\ .
	\end{equation}
\end{proposition}

\begin{proof}
	In case of $p_1\ge p_2$  the upper estimate for $\|\id_j\|$ follows from H\"older's inequality. We prove the lower estimate. Let us choose $\nu_0\in \{0, \dots, j\} $ such that
	\begin{equation}\label{12_04_1} 	\frac{\varphi_2(2^{\nu_0})}{\varphi_1(2^{\nu_0})}= \max_{0\le\nu\le j} \frac{\varphi_2(2^{\nu})}{\varphi_1(2^\nu)}.
	\end{equation} 
	  We  consider the  sequence  $\lambda=\{\lambda_k\}_k$ defined by 
	  \[\lambda_k= \begin{cases}
	  	1 & \text{if} \ Q_{0,k}\subset Q_{-\nu_0,0} ,\\
	  	0 & \text{otherwise} .
	  \end{cases} 
	  \]
	  Then $\|\lambda|\mpbet\|= \varphi_1(2^{\nu_0})$ and  $\|\lambda|\mpbzt\|= \varphi_2(2^{\nu_0})$. This proves \eqref{1503-0}. 
	  
	  Let now $p_1<p_2$  and $\|\lambda|\mmbet\|=1$. 
   Then $|\lambda_k|\le 1$ for any $k\in\Zn$ and thus
	\[\sum_{k:Q_{0,k}\subset Q_{-\nu,m}}\!\!\! |\lambda_k|^{p_2} \le  \sum_{k:Q_{0,k}\subset Q_{-\nu,m}}\!\!\! |\lambda_k|^{p_1} .\]
	In consequence,  for any $\nu$ with $0\le \nu\le j$, we have 
	\begin{align*}
		\frac{\varphi_2(2^\nu)^{p_2}}{2^{\nu d}} \sum_{k:Q_{0,k}\subset Q_{-\nu,m }} \!\!\! |\lambda_k|^{p_2} & \le  \;\frac{\varphi_2(2^\nu)^{p_2}}{2^{\nu d}}   \sum_{k:Q_{0,k}\subset Q_{-\nu,m }}\!\!\! |\lambda_k|^{p_1} \\
		&\le  \frac{\varphi_2(2^\nu)^{p_2}}{\varphi_1(2^\nu)^{p_1}} \  \|\lambda|\mmbet\|^{p_1} \le \frac{\varphi_2(2^\nu)^{p_2}}{\varphi_1(2^\nu)^{p_1}}.
	\end{align*}
	This proves the upper estimate in \eqref{1503-a}. 
	
	To prove the lower estimates in \eqref{1503-a} we once more choose $\nu_0\in \{0, \dots, {j}\}$ as in \eqref{12_04_1}. We may assume that $\nu_0$ is the smallest such  number.  Then $\frac{2^{d\nu_0}}{\varphi_1(2^{\nu_0})^{p_1}}\ge 1$ since $\nu_0\ge 0$ and $\varphi_1\in \Gpx{p_1}(\mathbb{D})$. We put 
	\[ N_0=\whole{\frac{2^{d\nu_0}}{\varphi_1(2^{\nu_0})^{p_1}}} . \]
	Then $1\le N_0 \le 2^{jd}$. We consider a sequence $\lambda=\{\lambda_k\}_{k\in {\mathcal{K}_j}}$ that consists of elements $\lambda_k \in \{0,1\}$ only, where exactly $N_0$ elements are equal to $1$ and  the rest are equal to zero. We assume that $\lambda_k=0$ if $Q_{0,k}\nsubseteq Q_{-\nu_0,0}$. Now we describe how the elements $\lambda_k=1$ are distributed inside the cube $Q_{-\nu_0,0}$.

	Let $\mathcal{N}_1=\{n\in \Z^d: Q_{-(\nu_0-1),{n}}\subset Q_{-\nu_0,0}\}$. We choose  the natural numbers $N_i^{(1)}$, $i=1,\ldots , 2^\nd$, in such a way that 
	\begin{equation}\label{N1}
		\sum_{n\in \mathcal{N}_1} N^{(1)}_n=N_0 \qquad\text{and}\quad  \whole{2^{-d}N_0}\le N_n^{(1)}\le \whole{2^{-d}N_0}+1.
	\end{equation} 
	We distribute the elements of the sequence among these cubes in such a way that the cube $Q_{-(\nu_0-1),n}$ contains $N_n^{(1)}$ elements equal to $1$ if $n\in \mathcal{N}_1$. 
	Then 
	\begin{equation}\label{N1-2}
		N^{(1)}_n\le \frac{2^{d(\nu_0-1)}}{\varphi_1(2^{\nu_0})^{p_1}}  +1 \le  2 \frac{2^{d(\nu_0-1)}}{\varphi_1(2^{\nu_0-1})^{p_1}} 
	\end{equation}
	since $\varphi_1\in \mathcal{G}_{p_1}(\mathbb{D})$. Consequently, 
        \begin{align}
          \sup_{{n:\, Q_{-(\nu_0-1),n}}\subset Q_{-j,0}} & \!\!
	   	\varphi(2^{\nu_0-1})\bigg(|Q_{-(\nu_0-1),n}|^{-1}\sum_{k:Q_{0,k}\subset Q_{-\nu_0-1,n}}\!\!|\lambda_k|^{p_1} \bigg)^{\frac{1}{p_1}}\nonumber \\
& \le	   	\varphi(2^{\nu_0-1}) 2^{-d\frac{\nu_0-1}{p_1}} \sup_{n\in \mathcal{N}_1} (N^{(1)}_n)^{1/p_1} \le  2^{1/p_1}  
	   \end{align}
   
   In the second step we distribute the elements of the sequence inside any cube $Q_{{-(\nu_0-1)}, n}\subset Q_{{-\nu_0}, 0}$.  This means  we distribute in a similar way as above $N^{(1)}_n$ unit elements  among the cube of side length $2^{\nu_0-2}$ that are contained  in  the  cube $Q_{-(\nu_0-1), n}$ of side length $2^{\nu_0-1}$, $n\in {\mathcal N}_1$. So we should distribute $N^{(1)}_n$ unit elements. In consequence such a cube $Q_{-(\nu_0-2),m}$ contains $N^{(2)}_m$ unit elements and  
\begin{align*}
	N^{(2)}_m  & \le \whole{2^{-d} N^{(1)}_n} +1 \ \le \ 2^{-d} (2^{-d}N_0 +1) +1\\ & \le   \frac{2^{d(\nu_0-2)}}{\varphi_1(2^{\nu_0})^{p_1}}  + 2^{-d}  +1 \ \le \ c \frac{2^{d(\nu_0-2)}}{\varphi_1(2^{\nu_0-2})^{p_1}}  , 
\end{align*}
cf. \eqref{N1} and \eqref{N1-2}. We continue this procedure along any descending path up to the moment when $N^{(i)}_m=1$.  At any  $(i)$-level we have  
 \begin{equation}
 	N^{(i)}_m\le \
 	c \ \frac{2^{d(\nu_0-i)}}{\varphi_1(2^{\nu_0-i})^{p_1}}
 \end{equation}
and we can choose the same   constant $c>0$ at  any level $(i)$. Moreover,  
 \begin{align*}
	\sup_{(i,m):\,Q_{-(\nu_0-i),m}\subset Q_{-j,0}}\!\!
	& \varphi(2^{\nu_0-{i}})\left(|Q_{-(\nu_0-i),m}|^{-1}\sum_{k:Q_{0,k}\subset Q_{-(\nu_0-i),m}}\!\!|\lambda_k|^{p_1} \right)^{\frac{ 1}{ p_1}} \\
& \le	\varphi(2^{\nu_0-i}) 2^{-d\frac{\nu_0-i}{p_1}} \sup_{n\in \mathcal{N}_1} (N^{(i)}_n)^{1/p_1} \le  c^{1/p_1}.   
\end{align*}
In consequence the norm of the sequence $\lambda$ in the space $\mpbet$ is bounded, 
\[ \|\lambda|\mpbet\| \le c^{1/p_1} .\] 
On the other hand, 
\begin{align*}
\frac{\varphi_2(2^{\nu_0})}{\varphi_1(2^{\nu_0})^{p_1/p_2}}\le C 
\|\lambda| \mpbzt\|
\le c^{1/p_1} \|\id\|. 
\end{align*}
This proves the result.
    \end{proof}

      \begin{remark}
      Proposition~\ref{lemma15030} coincides in case of $\varphi_i(2^k) \sim 2^{k\nd/u_i}$,
$k\in\no$, $0<p_i\leq u_i<\infty$, $i=1,2$, with \cite[Lemma~6.3]{HaSk_Mo_seq}. In fact, our above finding, sloppily written as 
\[
  \Big\|	\id_j: \mpbet\hookrightarrow \mpbzt\Big\| \ \sim \ 	\max_{0\le\nu\le j} \frac{\varphi_2(2^{\nu})}{\varphi_1(2^\nu)^{\varrho}},
  \]
with (hidden) constants independent of $j\in\no$ and $\varrho=\min\{1,\frac{p_1}{p_2}\}$, looks much nicer than the previous result for the  special case $\varphi_i(2^k) \sim 2^{k\nd/u_i}$, explicitly stated for all admissible situations of $p_i$ and $u_i$, $i=1,2$.  
\end{remark}

\begin{remark}
  The above finite-dimensional norm estimates turned out to be an essential key for our compactness studies of $\id_j$ (or even, closely connected, compact embeddings $\id_\Omega$ of smoothness Morrey spaces of type $\mathcal{M}_{u,p}(\Omega)$ on bounded domains $\Omega\subset\rd$) -- always 
 restricted to the case of $\varphi(t)\sim t^{\nd/u}$, $t>0$, $0<p<u<\infty$. In \cite{HaSk-morrey-comp} we studied their corresponding entropy and approximation numbers, $e_k(\id_j)$ and $a_k(\id_j)$, or $e_k(\id_\Omega)$ and $a_k(\id_\Omega)$, respectively. For details and properties of these concepts we refer to \cite{CS,EE,Koe,Pie-s}, as well as to \cite{Pia} for the more general setting of $s$-numbers. 
\end{remark}

    \section{Strictly singular embeddings}\label{sec-4}
    We return to the infinite-dimensional sequence spaces $\mps$, $\varphi\in\Gpd$, $0<p<\infty$, and their embeddings as characterised in Theorem~\ref{cont}. Recall that 
\begin{equation}
  \id : \mpse\hra\mpsz
\end{equation}
is never compact, in contrast to their finite-dimensional counterparts $\id_j: \mpbet\hookrightarrow \mpbzt$ considered in the previous section. Here $0<p_i<\infty$,  $\varphi_i\in\Gpx{i}{(\mathbb{D})}$, $i=1,2$. Now we study the question  when this embedding $\id $ is strictly singular. This notion goes back to Kato \cite{kato-58}. We refer to \cite[Sect.~5.2.3.8, p.201]{pie-history} for further historic information and \cite{Pie-2001} for some general relation between compact and strictly singular operators.

\begin{definition}
A bounded linear operator $T\in\mathcal{L}(X,Y)$ between two Banach spaces $X$ and $Y$ is said to be {\em strictly singular}, if there is no infinite-dimensional closed subspace $Z$ of $X$ such that $T:Z\to T(Z)$, the restriction of $T$ to $Z$, is an isomorphism.  
\end{definition}

\begin{remark}\label{def-ext}
Dealing with quasi-Banach spaces, we shall rely on the following alternative definition, see \cite[Def.~2.4]{lang-nek-23}. The linear and bounded operator $T:X\to Y$ acting between the quasi-Banach spaces $X$ and $Y$ is said to be {\em strictly singular} if for each infinite-dimensional subspace $Z\subset X$,
\begin{equation}
\inf \{ \|Tx|Y\|: x\in Z, \|x|X\|=1\}=0.
\end{equation}
\end{remark}

Let us only mention a few important features, otherwise we refer to \cite[Section~4.5]{Abra-Ali} for further details.      Clearly any compact operator is strictly singular, while the converse is not true, cf. \cite[Thm.~4.58]{Abra-Ali}.

\begin{proposition}\label{lp-strictly}
  \begin{enumerate}[\upshape\bfseries (i)]
  \item
    The composition of a strictly singular operator and a bounded linear operator is strictly singular.
  \item
    If $0< p<q\leq\infty$, then the embedding $\ \id: \ell_p\to\ell_q\ $ is a non-compact strictly singular operator.
  \item
    Let $0 < p <\infty$, $0 < q < r \leq\infty$. Then the embedding $\ \id: \ell_{p,q} \hra \ell_{p,r}\ $
is strictly singular.
  \end{enumerate}
\end{proposition}

For part (i) we refer to \cite[Cor.~4.62]{Abra-Ali}, for (ii) see \cite[Thm.~4.58]{Abra-Ali}, extended by (i) to $p<1$ and $q=\infty$, while part (iii) was recently obtained in \cite[Thm.~3.3]{lang-nek-23}. Here $\ell_{p,q}$ denote the Lorentz sequence spaces, recall \eqref{lorentz-1}, \eqref{lorentz-2}.\\

We begin with some preparation. Recall that for $\varphi \in \Gpd$ our assumption $\varphi(1)=1$ and the $\Gp$ condition imply that $\lim_{\nu\to\infty} 2^{-\nu\nd/p} \varphi(2^{\nu}) \in [0,1]$ exists.

\begin{lemma}\label{L1}
	Let $\varphi\in \Gpd$ and $\lim_{\nu\to\infty} 2^{-\nu\nd/p}\varphi(2^\nu) = 0$. For any sequence $\lambda = \{\lambda_m\}_{m\in\Zn}\in \ell_p(\Zn)$ there exists a dyadic cube $Q_\lambda$ such that 
	\begin{equation}
		\|\lambda|\mps\| = \frac{\varphi(\ell(Q_\lambda))}{|Q_\lambda|^{1/p}} \Big(\sum_{m: Q_{0,m}\subset Q_\lambda} |\lambda_m|^p\Big)^{1/p}. 
	\end{equation}  
\end{lemma}

\begin{proof}
It is sufficient to prove the statement for sequences $\lambda = \{\lambda_m\}_{m\in\Zn}$ with  $\|\lambda| \ell_p\|=1 $. In view of our assumption and $\varphi\in\Gpd$, for any  $0<\varepsilon<1$ there exists a number $\nu_0=\nu_0(\varepsilon,\lambda)>0$ 
such that 
\[  \frac{\varphi(2^\nu)}{2^{\nu d/p}}\le \varepsilon \qquad\text{for all}\quad \nu\geq \nu_0. \]
Moreover, there exists  a dyadic cube $\widetilde{Q}$ such that 
\[ \Big(\sum_{Q_{0,m}\nsubseteq \widetilde{Q}}|\lambda_m|^p\Big)^{1/p}\le \varepsilon . 
\]   
Thus, if a dyadic cube $Q$ is sufficiently large, $\ell(Q)\ge 2^{\nu_0}$, or $Q\nsubseteq \widetilde{Q}$, then 
\[\frac{\varphi(\ell(Q))}{|Q|^{1/p}} \Big(\sum_{Q_{0,m}\subset Q} |\lambda_m|^p\Big)^{1/p}\le \varepsilon^2<\varepsilon .\]    
In consequence, the supremum defining the norm of $\lambda$ in  the space
$\mps$  is taken over a finite family of dyadic cubes only. This implies the claim. 
\end{proof}

\begin{proposition}\label{P-lp-mps-ss}
Let $\varphi\in \Gpd$, and $\lim_{k\to\infty} 2^{-k\nd/p} {\varphi(2^k)}= 0$, 
$1\le p<\infty$. Then the embedding $\ell_p$ into $\mps$ is strictly singular. 	
	\end{proposition}

\begin{proof}
	One can easily see that the norm of the embedding $\ell_p \hookrightarrow \mps$ is equal to $1$, recall also Corollary~\ref{lp-mps}. 
	 
		Let $X$ be an infinite-dimensional subspace of $\ell_p$. Please note that if $Q$ is a dyadic cube of size $\ell(Q)\ge 1$, then the space 
		\[ X_Q= \{\lambda\in X: \lambda_m=0 \quad\text{if} \quad Q_{0,m}\subset Q \}\] 
		is also infinite-dimensional.  To prove that the embedding is strictly singular, it is sufficient to find a sequence $x^{(n)} \in X$ such that $\|x^{(n)}|\ell_p\|\ge c>0$ and $\|x^{(n)}|\mps\|\rightarrow 0$ if $n\rightarrow \infty$, cf.   
         Remark~\ref{def-ext}. 
         
Let $0<\varepsilon <1$. We choose an element  $\lambda^{(1)}=\{\lambda^{(1)}_m\}_{m\in\Zn} \in X$ with unit norm, $\|\lambda^{(1)}|\ell_p\|=1$. 
Parallel to the proof of Lemma~\ref{L1} we can find a dyadic cube $Q_1$ such that
\begin{align*}
	\Big(\sum_{Q_{0,m}\nsubseteq Q_1} |\lambda^{(1)}_m|^p\Big)^{1/p}\le \frac{\varepsilon}{2}  \qquad \text{and}\qquad 
	\frac{\varphi(\ell(Q_1))}{|Q_{1}|^{1/p}}< \frac{1}{2}. 
\end{align*}  

Afterwards we choose $\lambda^{(2)}=\{\lambda^{(2)}_m\}\in X_{Q_1}$    with $\|\lambda^{(2)}| \ell_p\|=1$. We can find a dyadic cube $Q_2$, $Q_1\subset Q_2$, such that
\begin{align*}
  \Big(\!\!\sum_{Q_{0,m}\nsubseteq Q_2} |\lambda^{(2)}_m|^p\Big)^{1/p}\le \frac{\varepsilon}{2^2},  
  \quad \dist(Q_1,\partial Q_2)\ge 1,\quad\text{and} \quad\frac{\varphi(\ell(Q_2))}{|Q_{2}|^{1/p}}< \frac{1}{2^2}.  	 
\end{align*}  

Inductively,  for any $k\in\nat$, we can find a sequence  $\lambda^{(k)}=\{\lambda^{(k)}_m\}_{m\in\Zn}\in X_{Q_{k-1}}$    with $\|\lambda^{(k)}| \ell_p\|=1$, and  a dyadic cube $Q_k$, such that $Q_1\subset \ldots \subset  Q_{k-1}\subset Q_k$, and  
\begin{align*}
  \Big(\sum_{Q_{0,m}\nsubseteq Q_k} |\lambda^{(k)}_m|^p\Big)^{1/p}\le \frac{\varepsilon}{2^k},  
  \quad  \dist(Q_{k-1},\partial Q_k)\ge 1,\quad \text{and}\quad
	\frac{\varphi(\ell(Q_k))}{|Q_{k}|^{1/p}}< \frac{1}{2^k}.
\end{align*}  

We define $\ x^{(n)} := \ds\frac{1}{n^{1/p}} \sum_{j=1}^n \lambda^{(j)} $, $n\in\nat$. Then $x^{(n)}\in X$ and    
 \begin{align*}
	\|x^{(n)}|\ell_p\| & =  \frac{1}{n^{1/p}} \Big(\sum_{k=1}^\infty\; \sum_{m:Q_{0,m}\subset Q_k\setminus Q_{k-1}} |\lambda^{(1)}_m+\ldots + \lambda^{(n)}_m|^p   \Big)^{1/p}\nonumber\\ 
	&\ge  \frac{1}{n^{1/p}} \Big(\sum_{k=1}^n \;\sum_{m:Q_{0,m}\subset Q_k\setminus Q_{k-1}} |\lambda^{(1)}_m+\ldots + \lambda^{(n)}_m|^p   \Big)^{1/p} .\nonumber
\end{align*}
But  
	\begin{align*}
	&  \Big(\sum_{m:Q_{0,m}\subset Q_k\setminus Q_{k-1}} |\lambda^{(1)}_m+\ldots + \lambda^{(n)}_m|^p   \Big)^{1/p}\\  
	& \geq\ \Big( \sum_{Q_{0,m}\subset Q_k\setminus Q_{k-1}} |\lambda^{(k)}_m|^p\Big)^{1/p} -  \Big(\sum_{Q_{0,m}\subset Q_k\setminus Q_{k-1}} |\lambda^{(1)}_m+\ldots + \lambda^{(k-1)}_m|^p  \Big)^{1/p} \\
	& \geq\ 1- \frac{\varepsilon}{2^k} - \sum_{j=1}^{k-1} \frac{\varepsilon}{2^j}\ge 1-\varepsilon\nonumber\\
\end{align*}
Thus we arrive at 
\begin{equation}
	\|x^{(n)}|\ell_p\| \ge 1-\varepsilon .
\end{equation}

 On the other hand, let $Q$ be an arbitrary dyadic cube of size at least $1$.  Then both cubes, $Q$ and $Q_n$ are dyadic,   therefore we should consider the following three cases: \quad (i) $Q\cap Q_n =\emptyset$, \quad  (ii) $ Q\subset Q_n$, \quad (iii) $Q_n\subset  Q$.\\

In case (i),
\begin{align} 
  \Big(\sum_{Q_{0,m}\subset Q} |x^{(n)}_m|^p\Big)^{1/p} & \le  n^{-1/p} \sum_{j=1}^n \Big(\sum_{Q_{0,m}\subset Q}  |\lambda^{(j)}_m|^p\Big)^{1/p}  \nonumber\\
  & \le n^{-1/p} \sum_{j=1}^n \frac{\varepsilon}{2^j}\le  \frac{\varepsilon}{n^{1/p}}.
\label{m-norm1}\end{align}
So, by the assumptions on $\varphi\in\Gpd$ and $\ell(Q)\geq 1$,  
\begin{equation} \label{m-norm2}
	\frac{\varphi(\ell(Q))}{|Q|^{1/p}}\Big(\sum_{Q_{0,m}\subset Q} |x^{(n)}_m|^p\Big)^{1/p} \le  \frac{\varepsilon}{n^{1/p}}.
\end{equation}

Now we consider the case (ii). We can choose the smallest $k$, $1\le k\le n$, such that  $Q\subset Q_k$. Then $ Q\subset Q_k\setminus Q_{k-1}$ or $Q_{k-1}\subset Q\subset Q_k$. In the first case we have
\begin{align}
  \Big(\sum_{Q_{0,m}\subset Q} |x^{(n)}_m|^p\Big)^{1/p} & \le  n^{-1/p} \left( 1 + \sum_{j=1}^{k-1} \Big(\sum_{Q_{0,m}\subset Q}  |\lambda^{(j)}_m|^p\Big)^{1/p}   \right) \nonumber\\
  &\le n^{-1/p} \left( 1 + \varepsilon    \right),  
\label{m-norm3}\end{align}
since the sequence $\lambda^{(j)}$ equals zero on $Q$ if $j>k$, and their $\ell_p$ norm on $Q$ can be estimated by $\frac{\varepsilon}{2^j}$ if $j\le k-1$. Similar as above this leads to
\begin{equation} \label{m-norm4}
	\frac{\varphi(\ell(Q))}{|Q|^{1/p}}\Big(\sum_{Q_{0,m}\subset Q} |x^{(n)}_m|^p\Big)^{1/p} \le  \frac{1+\varepsilon}{n^{1/p}}.
\end{equation}
In the second case,  $Q_{k-1}\subset Q\subset Q_k$, we have 
\begin{align}\label{m-norm5}
	\Big(\sum_{Q_{0,m}\subset Q} |x^{(n)}_m|^p\Big)^{1/p} \le  n^{-1/p} \sum_{j=1}^k \Big(\sum_{Q_{0,m}\subset Q}  |\lambda^{(j)}_m|^p\Big)^{1/p}  
	\le  \frac{k}{n^{1/p}}.
\end{align}
But the assumptions $\varphi\in\Gpd$, $Q_{k-1}\subset Q$ and the construction of the $\{Q_k\}_k$ imply
\begin{equation}\label{m-norm6}
\frac{\varphi(\ell(Q))}{|Q|^{1/p}}\le\frac{\varphi(\ell(Q_{k-1}))}{|Q_{k-1}|^{1/p}}\le \frac{1}{2^{k-1}}, 
\end{equation}
so
\begin{equation} \label{m-norm7}
	\frac{\varphi(\ell(Q))}{|Q|^{1/p}}\Big(\sum_{Q_{0,m}\subset Q} |x^{(n)}_m|^p\Big)^{1/p} \le  \frac{k}{2^{k-1}n^{1/p}}.
\end{equation}

It remains to consider the case (iii), that is, $Q_n\subset Q$. Now,
\begin{align}
	\Big(\sum_{Q_{0,m}\subset Q} |x^{(n)}_m|^p\Big)^{1/p} & \le
	\Big(\sum_{Q_{0,m}\subset Q_n} |x^{(n)}_m|^p\Big)^{1/p} +
                                                                \Big(\sum_{Q_{0,m}\subset Q\setminus Q_n} |x^{(n)}_m|^p\Big)^{1/p}\nonumber\\
  &\le 
	  n^{1-\frac{1}{p}} + \frac{\varepsilon}{n^{1/p}},
\label{m-norm8}\end{align}
where { the first  sum on the left-hand side of the last inequality can be estimated in the same way as \eqref{m-norm5}, and  the estimate of the second sum follows from \eqref{m-norm2}.} Now, the estimate \eqref{m-norm6} with $k-1=n$ implies 
\begin{equation}\label{m-norm9}
		\frac{\varphi(\ell(Q))}{|Q|^{1/p}}\Big(\sum_{Q_{0,m}\subset Q} |x^{(n)}_m|^p\Big)^{1/p} \le  \frac{n+\varepsilon}{2^{n}n^{1/p}} .
\end{equation}
Hence, \eqref{m-norm2}, \eqref{m-norm4}, \eqref{m-norm7} and \eqref{m-norm9} lead to the desired result, 
\[\|x^{(n)}|\mps\|\rightarrow 0 \qquad \text{if}\qquad n\rightarrow \infty.
\]
\end{proof}

\begin{corollary}
  Let $1\le  p<u<\infty$. Then the embedding $\ell_p$ into $\ms$ is strictly singular. 	
\end{corollary}

\begin{proof}
We apply Proposition~\ref{P-lp-mps-ss} with $\varphi(t)\sim t^{\nd/u}$.
\end{proof}

Now we return to our embedding $\id: \mpse \hra \mpsz$ and investigate its strict singularity, as announced in the beginning of the section.

\begin{theorem}\label{strictly}
  Let    ${0<} p_i<\infty$ and  $\varphi_i\in\Gpx{p_i}(\mathbb{D})$, with $\varphi_i(1)=1$, $i=1,2$. Assume $\sup_{k\in\no} \varphi_i(2^k)=\infty$, $i=1,2$, and let \eqref{cont-mps} be satisfied with  $\varrho=\min\{1,\frac{p_1}{p_2}\}$. We consider the continuous embedding
\begin{equation}\label{ss}
\id:	\mpse\hookrightarrow \mpsz .
\end{equation}
\begin{enumerate}[\upshape\bfseries (i)]
	\item
          {Let $\lim_{\nu\to\infty} 2^{-\nu\nd/p_2} \varphi_2(2^{\nu})=0$. Then $\id$ is strictly singular if $$\lim_{\nu\to\infty} 2^{-\nu\nd/p_1} \varphi_1(2^{\nu})=c>0,$$
            with the additional assumption that $p_i\ge 1$ in the case $p_1=p_2=r_{\varphi_2}$. \\
Conversely, if $\id$ is strictly singular, then  $\lim_{\nu\to\infty} 2^{-\nu\nd/p_1} \varphi_1(2^{\nu})=c>0$, that is,  $\mpse= \ell_{p_1}$.} 
	\item
	Assume that $\lim_{\nu\to\infty} 2^{-\nu\nd/p_2} \varphi_2(2^{\nu})=c>0$, {that is,  $\mpsz= \ell_{p_2}$}.
	Then $\id$ is strictly singular if, and only if, { $\lim_{\nu\to\infty} 2^{-\nu\nd/p_1} \varphi_1(2^{\nu})=c>0$ and $p_1<p_2$}.
\end{enumerate}
  \end{theorem}

\begin{proof}
{\em Step 1}. We deal with (i). {We} prove that if  $\lim_{\nu\to\infty} 2^{-\nu\nd/p} \varphi(2^{\nu})=0$, then the space $\mps$, $0<p<\infty$,  contains an isometric copy of $\ell_\infty$. Let us choose a strictly increasing sequence $j_k\in \N$ such that
\begin{equation}
 2^{-j_k\nd/p} \varphi(2^{j_k})\le k^{-{\frac{1}{p}}}, \quad k\in\N.
\end{equation}	 
We define a subspace $X$ of $\mps$ in the following way 
\begin{align*}
  X=\bigg\{&\lambda=(\lambda_m)_{m\in \Z^d}\in \mps: \\
&\lambda_m = 
\begin{cases}
	\mu_k & \text{if} \ m=(2^{j_k}-1,0\ldots ,0),\\
	0 & \text{otherwise}, 
      \end{cases} \quad\text{for some}\; \mu=(\mu_k)\in \ell_{\infty}(\N) \bigg\}.
\end{align*}
{Then $Q_{-j_k,0}$ is the smallest dyadic cube that can contain   $k$, $k>1$, non-zero elements of the sequence $\lambda$.} So, if $j_k\le j <j_{k+1}$, then 
\begin{align*} \sup_{j\in\N: j_k\le j < j_{k+1} ; m\in \Zn}& \frac{\varphi(\ell(Q_{-j,m}))}{|Q_{-j,m}|^{1/p}}  \Big( \sum_{{\nu}:\, Q_{0,{\nu}}\subset Q_{-j,m}} \!\!\!|\lambda_{\nu}|^p\Big)^\frac{1}{p}\\
  &\le 
{k^\frac{1}{p}} \|\lambda|\ell_{\infty}\|\frac{\varphi(\ell(Q_{-j_k,m}))}{|Q_{-j_k,m}|^{1/p}}\le \|\lambda|\ell_{\infty}(\zd)\|=\|\mu|\ell_{\infty}(\nat)\|.  
\end{align*} 
Consequently,  
\begin{align*}
  \|\mu|\ell_{\infty}(\nat)\| & \le \|\lambda|\mps(\zd) \| \\
  & =  \sup_{j\in\no; m\in \Zn} \varphi(\ell(Q_{-j,m}))  \Big(|Q_{-j,m}|^{-1} \sum_{k:\, Q_{0,k}\subset Q_{-j,m}} \!\!\!|\lambda_k|^p\Big)^\frac{1}{p} \Big\}\\ & \le 	\|\mu|\ell_{\infty}(\nat)\|. 
\end{align*}

If $\varphi_1$ and $\varphi_2$  are such that $\lim_{\nu\to\infty} 2^{-\nu\nd/p_i} \varphi_i(2^{\nu})=0$, $i=1,2$,  then one can choose a sequence $j_k$  such that 
\begin{equation}
	2^{-j_k\nd/p_i} \varphi_i(2^{j_k})\le k^{-1} , \qquad k\in\nat, \quad i=1,2. 
\end{equation}	 
Then  the space $X$ is a subspace of $\mpse$ and $\mpsz$. Moreover,  
if $\lambda\in X$, then $\|\lambda|\mpse\|= 	\|\lambda|\ell_{\infty}\|= 	\|\lambda|\mpsz\|$. So $\id$ is not strictly singular in that case.
 
{\em Step 2.}   Let now $\lim_{\nu\to\infty} 2^{-\nu\nd/p_2} \varphi_2(2^{\nu})=0$ and $\lim_{\nu\to\infty} 2^{-\nu\nd/p_1} \varphi_1(2^{\nu})=c>0$.  This assumption  implies $\mpse = \ell_{p_1}$. Combining Corollary~\ref{lp-mps} and Proposition~\ref{lp-strictly}, we obtain for $p_1<p_2$ that
\[ \id : \mpse = \ell_{p_1} \hra \ell_{p_2} \hra \mpsz \]
is strictly singular. 

\noindent
If $p_1=p_2 <r_{\varphi_2}$, then 
we benefit from Proposition~\ref{lp-strictly},
regarding
\begin{equation}\label{ss2}
\id: \mpse = \ell_{p_1} = \ell_{p_2} \hra \ell_{r_{\varphi_2}}\hra m_{\varphi_2,p_2}, 
\end{equation}
and since the last but one embedding is strictly singular, so it is $\id$. 

\noindent
Analogously, if $p_1=p_2=r_{\varphi_2}\ge 1$, then  Proposition~\ref{P-lp-mps-ss} implies that  the embedding
\begin{equation}\label{ss2a}
		\id: \mpse = \ell_{p_1} \hra  m_{\varphi_2,p_1}, 
	\end{equation}
is strictly singular. 

\noindent 

It remains to deal with $p_2< p_1$. Then $\varphi_1\in\Gpx{p_1}\subset \Gpx{p_2}$ and $\rphi{\varphi_1}\geq p_1>p_2$.  So
	\[ \id: \mpse=\ell_{p_1} \hra \ell_{\rphi{\varphi_1},\infty} \hra m_{\varphi_1, p_2} \hra \mpsz, \]
	where we benefited from $p_2<\rphi{\varphi_1}$, Lemma \ref{lorentz-mps} and  Proposition~\ref{Prop-1-mps}(iii). Now the first embedding   is strictly singular, cf. Proposition~\ref{lp-strictly}, so it is $\id$.

{\em Step 3}. {It remains to consider the case   $\lim_{\nu\to\infty} 2^{-\nu\nd/p_2} \varphi_2(2^{\nu})=c>0$, i.e., 
$m_{\varphi_2,p_2}=\ell_{p_2}$. It follows from Corollary \ref{corlp} that if the embedding \eqref{ss} is continuous, then $\mpse=\ell_{p_1}$. So this case coincides with the classical statement about embeddings of $\ell_p$ spaces. } \end{proof}

We finally consider the special case $\varphi_i(t) \sim t^{\nd/u_i}$, $i=1,2$, and explicate Theorem~\ref{strictly}. \\

\begin{corollary}\label{mup-strictly}
  Let    $0< p_i \leq u_i<\infty$, $i=1,2$, and
  \[
\frac{u_1}{u_2}\leq \min\left(1, \frac{p_1}{p_2}\right).
  \]
We consider the continuous embedding
\[
\id_u:	\mse\hookrightarrow \msz .
\]

\begin{enumerate}[\upshape\bfseries (i)]
\item
  Assume that $p_2<u_2$. Then $\id_u$ is strictly singular if, and only if, $p_1=u_1$.
\item
  Assume that $p_2=u_2$. Then $\id_u$ is strictly singular if, and only if, $u_1=p_1<p_2$.

\end{enumerate}
In particular, $\id_u$ is not strictly singular whenever $p_1<u_1$.
    \end{corollary}

   \begin{remark}
      The above result obviously extends Proposition~\ref{lp-strictly} in view of Corollary~\ref{lp-mps} and Lemma~\ref{lorentz-mps}. Note that in the above setting the special situation $p_1=p_2=\rphi{\varphi_2}$ does not appear in part (i) of Corollary~\ref{mup-strictly}, since this means now $p_1=p_2=u_2$ different from Theorem~\ref{strictly}(i).
      \end{remark}

\end{document}